\documentclass[12pt]{article}
\usepackage[centertags,intlimits]{amsmath}                     
\usepackage{color}\usepackage{amsfonts}                     
\usepackage{amsthm}
\usepackage{graphics}
\allowdisplaybreaks[2]
\numberwithin{equation}{section}
\newtheorem{theorem}{Theorem}[section]
\newtheorem{lemma}{Lemma}[section]

\theoremstyle{remark}
\newtheorem{remark}{Remark}[section]

\providecommand{\abs}[1]{\lvert #1\rvert}
\providecommand{\norm}[1]{\lVert #1\rVert}

\newcommand{\nc}{\newcommand}
\nc{\vb}{\mathbf{v}}
\nc{\bx}{\mathbf{x}}
\nc{\by}{\mathbf{y}}
\nc{\bz}{\mathbf{z}}
\nc{\bu}{\mathbf{u}}
\nc{\bv}{\mathbf{v}}
\nc{\ba}{\mathbf{a}}
\nc{\bs}{\mathbf{s}}
\nc{\bq}{\mathbf{q}}
\nc{\bd}{\mathbf{d}}
\nc{\bb}{\mathbf{b}}
\nc{\bc}{\mathbf{c}}
\nc{\bi}{\mathbf{i}}
\nc{\bfr}{\mathbf{r}}
\nc{\bA}{\mathbf{A}}
\nc{\R}{\mathbb R}
\nc{\N}{\mathbb N}
\nc{\C}{\mathbb C}
\nc{\D}{\mathbb D}
\nc{\Z}{\mathbb Z}
\nc{\F}{\mathbf F}
\nc{\bbS}{\mathbb S}
\nc{\B}{\cal B}
\nc{\br}{\bigr}
\nc{\bl}{\bigl}
\nc{\Bl}{\Bigl}
\nc{\Br}{\Bigr}
\nc{\ind}{\mathbf{1}}
\nc{\bP}{\mathbf{P}}
  
\DeclareMathOperator*{\nt}{int}  
\DeclareMathOperator*{\cl}{cl}  

\title{Large deviation limits of invariant measures}
\author{Anatolii A. Puhalskii \footnote{Email: puhalski@iitp.ru}\\
 Institute for Problems in Information
Transmission}

\begin{document}

\maketitle
\begin{flushright}
\mbox{To the memory of my mother}
\end{flushright}
\sloppy
\vspace{1.cm}



\begin{abstract}
This paper is concerned with the general theme of
relating the Large Deviation Principle (LDP) for the invariant
measures of stochastic processes to the associated trajectorial
LDP. It is shown  that if  the trajectorial deviation function
has certain structure and if  the invariant measures are
 exponentially tight,
then the LDP for the invariant measures is implied by the trajectorial LDP,  no
other properties of the stochastic processes in question being material. 
As an application,
we obtain an LDP for the stationary distributions of jump diffusions. 
Methods of large deviation convergence and idempotent probability play
an integral part.
\end{abstract}
\begin{flushleft}
\mbox{
 }
\end{flushleft}

\section{Introduction}
\label{sec:introduction}

Let $X^n=(X^n_t\,,t\ge0)\,,n\in\mathbb N\,,$ be a sequence of 
rightcontinuous  stochastic processes with lefthand limits defined on a
complete probability space $(\Omega,\mathcal{F},\mathbf P)$ and taking
 values in a
 metric space $\mathbb S$\,. Suppose that the $X^n$ satisfy a trajectorial
 LDP in the Skorohod space 
$\mathbb D(\R_+,\mathbb S)$ with a (tight) deviation function
(also referred to as a rate function or an action functional)   $ I(X)\,,X\in\mathbb D(\R_+,\mathbb S)$\,. Let $ P^n$ represent  time-invariant
   distributions of the $X^n$ so that
   \begin{equation}
     \label{eq:1}
     \mathbf P(X_t^n\in \Gamma)=P^n(\Gamma)\,
   \end{equation}
for all Borel sets $\Gamma\subset \mathbb S$ and all $t\in\R_+$\,.
One seeks to obtain an LDP for the $P^n$ from that for the $X^n$\,.

A basic example is a diffusion in a finite--dimensional
Euclidean space $\R^d$ with small noise:
\begin{equation*}
  dX^n_t=b(X^n_t)\,dt+\frac{1}{\sqrt{n}}\,dW_t\,,X^n_0=x\,,
\end{equation*}
where $W_t$ represents a standard Wiener process.
Under regularity assumptions,
 the processes $X^n=(X^n_t\,,t\ge0)$  obey an LDP in $\mathbb D(\R_+,\R^d)$\,, 
as $n\to\infty$\,, for rate $n$ with
deviation function $ I(X)=(1/2)\int_0^\infty\abs{\dot 
  X_t-b(X_t)}^2\,dt $\,, provided $X=(X_t\,,t\ge0)$ is an absolutely continuous
function and $X_0=x$\,, and $ I(X)=\infty$\,, otherwise.
Let us assume, in addition, that the $X^n$ admit unique invariant
measures.
If the differential equation 
\begin{equation}
  \label{eq:15}
\dot X_t=b(X_t)
\end{equation}
 has a unique equilibrium which is asymptotically
stable, then, in some generality,
 the invariant measures satisfy 
 an LDP in $\R^d$ with the deviation
function that is the quasipotential $V(x)=\inf_{t\ge0}
 \inf_{X:\,X_0=O\,,X_t=x} I(X)$\,, $O$ representing the
 asymptotically stable equilibrium,  see Freidlin and Wentzell
\cite{wf2}. 
 Things are drastically different where the differential equation
\eqref{eq:15} has multiple attractors. A quasipotential is no longer
the correct answer.
In a remarkable accomplishment,
 Freidlin and Wentzell 
 \cite{wf2} identified the deviation function for that setup.
 Their ingenious analysis relied heavily
on the strong Markov property and involved an intricate study of attainment
times. 

The purpose of this contribution is to formulate a general
framework that should enable one to derive this sort of result. We  show
that once a trajectorial LDP has been established, an
LDP for an exponentially tight collection of
 invariant measures   can be inferred from 
the properties of the trajectorial deviation function without invoking
 the particulars of the stochastic processes in question. 
The compactness approach developed
 in Puhalskii \cite{Puh97,Puh01} is used.
  Rather than  checking the lower and upper bounds in the definition of
  the LDP,
one   proves
exponential tightness, first, and, then, capitalising on relative
compactness delivered by the exponential tightness, attempts to
 identify the deviation function through  equations
that arise as large deviation limits of equations
satisfied by  the original 
stochastic processes.  
For instance, in the large deviation limit,
 exponential martingales turn  into
''maxingales'' and, in analogy with a  martingale problem, one may look for a deviation function that renders
certain functions of trajectories maxingales. 
Since  large deviation limits take one into
 the rhelm of tropical (or idempotent) mathematics, with
the field of reals being replaced with a tropical semifield so that
subtraction is no longer available, the limit equations may be less
informative and satisfied by a variety of deviation functions.  
 The challenge is to come up with 
equations  that specify the deviation function uniquely.
Vis a vis invariant measures, this approach was applied in
Puhalskii \cite{Puh03,Puh19a}. There, a large deviation limit was taken in
the definition of an invariant measure. 
In the limit, an invariant measure turns into
an invariant deviation function. (As a matter of fact, we prefer
dealing with negative exponentials of deviation functions which are
 akin to probabilities, with  maxima being substitutes for sums, and which we
dub ''deviabilities''.) In Puhalskii \cite{Puh03,Puh19a}, we were unable to
tackle the case of a multitude of equilibria, the ''naive'' limit
being too crude.
The insight of this paper is to pass to 
 large deviation limits  in balance  equations 
 for probability fluxes  across
  cuts in the state space.
 In the limit,  max balance equations are obtained, according 
to which
 max fluxes across
cuts balance.
Those equations are shown to identify the limit deviability.
As an application,  an LDP for the invariant
measures of jump diffusions
 is established. Besides, Freidlin and Wentzell's \cite{wf2} prescription for
calculating the deviation function is extended 
to the case where the limit
differential equation has  infinitely many equilibria. 

This paper is structured as follows.
In Section \ref{sec:setup-main-results}, the
 approach is outlined and the main results are stated.
Section \ref{sec:large-devi-conv} contains
the derivation of max balance equations and a proof that they identify
the limiting deviability.  Section \ref{sec:diff-with-jumps}
tackles the LDP for the invariant
measures of  jump diffusions.
 \section{The LDP for invariant measures}
\label{sec:setup-main-results}


Let us fix terminology and review necessary facts and definitions,
see, e.g., Puhalskii \cite{Puh01}.
Let $\mathbb E$ represent a metric space. 
 Let $\mathcal{P}(\mathbb{E})$ denote the
power set of $\mathbb{E}$. 
We say that  set function $\mathbf\Pi:\, \mathcal{P}(\mathbb{E})\to[0,1]$
is a deviability if 
$\mathbf\Pi(E)=\sup_{x\in E}\mathbf\Pi(\{x\}),\,E\subset \mathbb{E},$
where the function $\mathbf\Pi(x)=\mathbf\Pi(\{x\})$ is such that 
$\sup_{x\in \mathbb{E}}\mathbf\Pi(x)=1$ and the sets
$\{x\in \mathbb{E}:\,\mathbf\Pi(x)\ge \gamma\}$ are compact for all $\gamma \in(0,1]$. 
(One can also refer to $\mathbf\Pi$ as a maxi measure.
Suprema over null sets are defined to
equal 0.)
The function $\mathbf\Pi(x)$ is called a deviability density.
A deviability  is a tight set function in the sense that 
$\inf_{K\in\mathcal{K}(\mathbb E)}\mathbf\Pi(\mathbb S\setminus K)=0$\,,
 where $\mathcal{K}(\mathbb E)$ stands for the collection of compact
subsets of $\mathbb E$\,.
If $\Xi$ is a directed set and $F_\xi\,, \xi\in \Xi\,,$ is a net of closed
subsets of $\mathbb E$ that is nonincreasing with respect to the
partial order on $\Xi$ by inclusion, then $\mathbf\Pi(\cap_{\xi\in
  \Xi}F_\xi)=
\lim_{\xi\in \Xi}\mathbf\Pi(F_\xi)$\,.
  The continuous images of deviabilities are deviabilities, i.e., 
if $f:\,\mathbb E\to\mathbb E'$ is continuous, with $\mathbb E'$ being
a metric space, then $\mathbf\Pi\circ f^{-1}$ defined by
$\mathbf\Pi\circ f^{-1} (E')=\mathbf\Pi( f^{-1}(E'))$ is a deviability on $\mathbb
E'$\,, where $E'\subset \mathbb E'$\,.

We say that  a sequence $Q_n$
of probability measures on the Borel $\sigma$--algebra of $\mathbb E$
  Large
Deviation (LD)  converges at rate $n$ to  deviability $\mathbf\Pi$  if
for every bounded continuous non-negative function $f$ on $\mathbb{E}$\,,
\[
\lim_{n\to\infty}\left(\int_{ \mathbb{E}}f(x)^{n}\,Q_n(dx)\right)^{1/n}=
\sup_{x\in \mathbb{E}}f(x)\mathbf\Pi(x).
\]
Equivalently, one may require that 
$\liminf_{n\to\infty}Q_n(\Gamma)^{1/n}\ge
\mathbf\Pi(\nt \Gamma)$ 
and $\limsup_{n\to\infty}Q_n(\Gamma)^{1/n}\le
\mathbf\Pi(\cl \Gamma)$ for every Borel set $\Gamma$\,.
(As is customary,  $\text{int}$ is used to denote the interior of
a set and  $\text{cl}$ is used  to denote the closure of a set.)
We say that the sequence $Q_n$ is exponentially tight of order $n$ if 
$\inf_{K\in\mathcal{K}(\mathbb E)}\limsup_{n\to\infty}Q_n(\mathbb E\setminus
K)^{1/n}=0$\,. If the sequence $Q_n$ is exponentially tight of order
$n$\,, then there exists  subsequence $Q_{n'}$ that LD
converges at rate $n'$ to a deviability.
 Any such deviability will be
referred to as a Large Deviation (LD) limit point of the $Q_n$\,.
It is immediate that $\mathbf\Pi$
is a deviability if and only if $I(x)=-\ln \mathbf\Pi(x)$ is a tight  deviation
function, i.e., 
 the sets $\{x\in\mathbb E:\,I(x)\le \gamma\}$ are compact for all
$\gamma\ge0$ and $\inf_{x\in\mathbb E}I(x)=0$\,, and that the $Q_n$ LD
converge to $\mathbf\Pi$ if and only if they
 obey the LDP for rate $n$ with deviation function
$I$\,, i.e.,  $\liminf_{n\to\infty}(1/n)\,\ln Q_n(G)\ge -\inf_{x\in
  G}I(x)$\,, for all open sets $G$\,, and
$\limsup_{n\to\infty}(1/n)\,\ln Q_n(F)\le -\inf_{x\in
  F}I(x)$\,, for all closed sets $F$\,.

We return to the setup of processes $X^n$ with trajectories in the Skorohod
space $\mathbb D(\R_+,\mathbb S)$ and invariant distributions $P^n$ and state our hypotheses.
\subparagraph{2.1}
There exist 
versions of regular conditional distributions $\mathbf P(X^n\in
\Gamma|X_0^n=x)$\,, where $\Gamma$ represents a Borel subset of 
$\mathbb
D(\R_+,\mathbb S)$ and $x\in\mathbb S$\,, such that whenever 
$x^n\to x$\,, as $n\to\infty$\,, with $x^n$ belonging to the support
of $P^n$\,,  the distributions
 $\mathbf P(X^n\in\Gamma|X_0^n=x^n)$\,, considered as probability measures on $\mathbb
D(\R_+,\mathbb S)$\,, LD converge at rate $n$ to deviability $\mathbf\Pi_x$
on $\mathbb
D(\R_+,\mathbb S)$ such that $ \mathbf\Pi_x(X)=0$
unless  $X=(X_t,\,t\in\R_+)\in\D(\R_+,\mathbb S)$ is a continuous function. (It is immediate that $
\mathbf\Pi_x(X)=0$ unless $X_0=x$\,.)

\subparagraph{2.2}
The function  $\mathbf\Pi_x(X)$ is upper semicontinuous in
$(x,X)$
and the set $\cup_{x\in K}\{X:\,\mathbf\Pi_x(X)\ge\gamma \}$ is compact, for
all $\gamma \in(0,1]$ and  $K\in \mathcal{K}(\mathbb S)$\,.
(The former condition is fulfilled if $\mathbf\Pi_x(X)=\Pi(X)$ provided $X_0=x$ and
$\mathbf\Pi_x(X)=0$\,, otherwise, where $\Pi(X)$ is an upper semicontinuous
 function.)

\subparagraph{2.3}
For all $X\in \mathbb D(\R_+,\mathbb S)$\,,
\begin{equation}
  \label{eq:3}
   \mathbf\Pi_{x}(X)= \mathbf\Pi_{x}(\pi_s^{-1}(\pi_sX)) \mathbf\Pi_{X_s}(\theta_sX)\,,
\end{equation}
where $\pi_sX=(X_t\,,t\in[0,s])$ and 
$\theta_sX=(X_{s+t}\,, t\ge0)$\,.
(One can see that it is a  version of the Markov property.)

Let
\begin{align}
  \label{eq:5}
  \Pi_{x,t}(y)=\sup_{X\in \D(\R_+,\mathbb S)\,:\,X_t=y} \mathbf\Pi_x(X)
  \intertext{ and }
\Pi_{x,t}(\Gamma)=\sup_{y\in\Gamma}\Pi_{x,t}(y)\,,
\Gamma\subset\mathbb S\,.\notag
\end{align}
It is noteworthy that $\Pi_{x,t}(y)$ is upper semicontinuous in
$(x,t,y)$ and is a deviability density in $y$ (as a continuous image
of $\mathbf\Pi_x$)\,. Let
 $d(\cdot,\cdot)$ represent the metric on $\mathbb S$\,.

\subparagraph{2.4}

There exists   set
$A
\subset \mathbb S$\,,
which is locally finite in the sense that
every compact subset of $\mathbb S$ contains at most finitely many of
the elements of $A$\,, such that the following
properties hold:
\begin{trivlist}
\item[(1)] if 
 $ \mathbf\Pi_x(X)=1$\,, then   $\inf_{t\ge0}d(X_t, A)=0$\,,
\item[(2)]
if $X_t= a$\,, for all $t\ge0$\,, where $a\in A$\,, then $\mathbf\Pi_a(X)=1$\,,
\item[(3)] if $a,a'\in A$\,, then $\Pi_{a,t}(a')>0$\,, for some $t>0$\,,
\item[(4)] for any $\epsilon>0$\,, there exists $\delta>0$ such
  that if $d(x,A)<\delta$\,, then 
 $\Pi_{x,s_0}(a)>1-\epsilon$
and $\Pi_{a,s_1}(x)>1-\epsilon$\,, for some  $s_0>0$\,,
  $s_1>0$\,, and $a\in A$\,,
\item[(5)] for any  $x\in \mathbb S$ and $\epsilon>0$\,, there exist $\delta>0$\,,
  $t_0$ and $t_1$
  such that 
$\Pi_{x,t_0}(x')>1-\epsilon$
and $\Pi_{x',t_1}(x)>1-\epsilon$ whenever $d(x,x')<\delta$\,. 
\end{trivlist}
(The set $A$ plays the role of the set of attractors of \eqref{eq:15}.)

\subparagraph{2.5}

The net $(\Pi_{x,t}(\Gamma)\,,\Gamma\subset\mathbb S)\,, t\ge0,$ is 
tight uniformly over $x$ from compact sets. More explicitly,
for any compact $K_1\subset \mathbb S$ and $\epsilon>0$\,, there
exists compact $K_2\subset \mathbb S$ such that
$  \limsup_{t\to\infty}\sup_{x\in K_1}\Pi_{x,t}(\mathbb S\setminus K_2)<\epsilon\,.
$

\vspace{.5cm}
We prove in Lemma \ref{le:attraction} that under these hypotheses,
  there exist the limits
\begin{equation*}
  \Pi(x,y)=\lim_{t\to\infty}\Pi_{x,t}(y)\,,
\end{equation*}
 the function $\Pi(x,y)$ is continuous in $(x,y)$ and
  is upper compact in $y$ uniformly over $x$
  from compact sets, the latter property  meaning
  that, for any $\gamma>0$ and compact $K$\,,
the set $\cup_{x\in K}\{y:\,\Pi(x,y)\ge \gamma\}$ is compact. 
\footnote{ A function $f$ is said to be upper compact if the sets
$\{x:\,f(x)\ge \lambda\}$ are compact. This definition is modelled on
 the definition of
a lower semicompact function   as a function
$f$ such that 
the sets $\{x:\,f(x)\le\lambda\}$ are relatively compact,
see, e.g.,  Aubin \cite{Aub93}. }
For sets $\Gamma\subset \mathbb S$ and $\Gamma'\subset \mathbb S$\,, we let
\begin{equation*}
  \Pi(\Gamma,\Gamma')=\sup_{\substack{x\in \Gamma,\\y\in \Gamma'}}\Pi(x,y)\,.
\end{equation*}
(The following notational convention is adhered
to.
If either $\Gamma$ or $\Gamma'$ is a one-element set, then we identify such a set with its only element, e.g., we define
$\Pi(x,\Gamma')=\Pi(\{x\},\Gamma')$\,.) 
It is noteworthy that
 $\Pi(a,a)=1$\,, for $a\in A$\,.

For sets $\Gamma$ and $\Gamma'$ and deviability $\mathbf\Pi$\,, the max flux from $\Gamma$ to
$\Gamma'$ relative to $\mathbf\Pi$ is defined by
\begin{equation*}
  \Phi_{\mathbf\Pi}(\Gamma,\Gamma')=\sup_{x\in \Gamma}\mathbf\Pi(x)\Pi(x,\Gamma')\,.
\end{equation*}

Max balance is said to hold for $\mathbf\Pi$  if
\begin{equation}
  \label{eq:56}
  \Phi_{\mathbf\Pi}(A',A'')=\Phi_{\mathbf\Pi}(A'',A')\,,\text{ for any partition
  $\{A',A''\}$ of $A$}.
\end{equation}
\begin{theorem}
  \label{the:vf}
Suppose that the sequence $P^n$ of time invariant distributions of
the processes $X^n$ is exponentially tight of order $n$\,.
Let $\mathbf\Pi$ represent an LD limit point of the $P^n$\,.
Let conditions  $\mathbf{2.1}$ -- $\mathbf{2.5}$ hold.
Then,
\begin{enumerate}
\item 
 $\mathbf\Pi$ satisfies the max balance  in \eqref{eq:56}
and the normalisation condition 
that $\mathbf \Pi(A)=1\,,$ which 
 requirements uniquely specify  the restriction of $\mathbf\Pi$ to $A$\,,
\item  For all $x\in\mathbb S$\,,
\begin{equation*}
   \mathbf\Pi(x)=\Phi_{\mathbf\Pi}(A,x)\,,
\end{equation*}
\item The $P^n$ LD converge to $\mathbf\Pi$\,.\end{enumerate}
 \end{theorem}
 \begin{remark}
An outline of the proof of \eqref{eq:56} is as follows. As $P^n$ is
invariant, balancing probability fluxes between $\Gamma\subset\mathbb
S$ 
and $\Gamma^c=\mathbb S\setminus \Gamma$  yields the equation
\begin{equation*}
        \int_\Gamma\mathbf P(X^n_t\in \Gamma^c|X^n_0=x)\,P^n(dx)=
\int_{\Gamma^c}\mathbf P(X^n_t\in \Gamma|X^n_0=x)\,P^n(dx)\,,
\end{equation*}
with details being provided in the proof of Lemma \ref{the:maxflow}.
By the LD convergence of the 
 $((\mathbf P(X^n\in\Delta|X_0^n=x^n),\,\Delta\subset \D(\R_+,\mathbb S))$ 
to $\mathbf\Pi_{x}$ in {\bf 2.1}, for ''nice'' sets
$\Gamma\subset\mathbb S$\,,
\begin{align*}
 \lim_{n\to\infty}\mathbf P(X^n_t\in \Gamma|X^n_0=x^n)^{1/n}=
\Pi_{x,t}(\Gamma)   \,.
\end{align*}
Also LD convergence of the $P^n$ to $\mathbf\Pi$ implies that
\begin{equation*}
  \lim_{n\to\infty}\bl(\int_{\mathbb S}f_n(x)^nP^n(dx)\br)^{1/n}
=\sup_{x\in\mathbb S}f(x)\mathbf\Pi(x)\,,
\end{equation*}
provided $f_n(x^n)\to f(x)$ whenever $x^n\to x$\,.
Hence,  with $f_n(x)=\mathbf P(X^n_t\in\Gamma|X^n_0=x)^{1/n}\ind_{ \Gamma^c}(x)$ and $f(x)=
\Pi_{x,t}(\Gamma)\ind_{\Gamma^c}(x)$\,, assuming $\Gamma^c$ is ''nice'',
\begin{equation*}
  \lim_{n\to\infty}\bl(\int_{\Gamma^c}\mathbf P(X^n_t\in
  \Gamma|X^n_0=x)\,P^n(dx)\br)^{1/n}=
\sup_{x\in \Gamma^c}\sup_{y\in \Gamma}\Pi_{x,t}(y)\mathbf\Pi(x)\,.
\end{equation*}
Similarly, once again assuming $\Gamma^c$ and $\Gamma$ ''nice'',
\begin{equation*}
    \lim_{n\to\infty}\bl(\int_{\Gamma}\mathbf P(X^n_t\in
  \Gamma^c|X^n_0=x)\,P^n(dx)\br)^{1/n}=
\sup_{x\in \Gamma}\sup_{y\in \Gamma^c}\Pi_{x,t}(y)\mathbf\Pi(x)\,.
\end{equation*}
Putting everything together and letting $t\to\infty$
 yields 
$   \Phi_{\mathbf\Pi}(\Gamma,\Gamma^c)=\Phi_{\mathbf\Pi}(\Gamma^c,\Gamma)\,.
 $
Furthermore, it is shown that, for
 $\Gamma=\{x\in\mathbb S:\,
\Phi_{\mathbf\Pi}(A',x)\ge\Phi_{\mathbf\Pi}(A'',x)\}$\,, 
in some generality, $\Phi_{\mathbf\Pi}(\Gamma,\Gamma^c)=\Phi_{\mathbf\Pi}(A',A'')$
and $\Phi_{\mathbf\Pi}(\Gamma^c,\Gamma)=\Phi_{\mathbf\Pi}(A'',A')$\,,
implying \eqref{eq:56}. \end{remark}
\section{LD convergence and max balance}
\label{sec:large-devi-conv}
In this section, we prove  the claims of Section  \ref{sec:setup-main-results}.
 Freidlin and Wentzell's 
solution to the max balance equations is recapitulated and extended.

Let us note that \eqref{eq:3} entails   the semigroup property that
\begin{equation}
  \label{eq:26}
      \Pi_{x,s+t}(y)=\sup_{z\in\mathbb S}\Pi_{x,s}(z)\Pi_{z,t}(y)\,.
  \end{equation}
Indeed, by \eqref{eq:3} and \eqref{eq:5},
\begin{multline*}
  \Pi_{x,s+t}(y)=\sup_{X:\,X_{s+t}=y}\mathbf\Pi_x(X)=
\sup_{z\in\mathbb S}\sup_{\substack{X:\,X_0=x,\\X_s=z,\,(\theta_sX)_{t}=y}}\;
\sup_{X'\in\pi_s^{-1}(\pi_sX)}\mathbf\Pi_x(X')\mathbf\Pi_z(\theta_sX)\\=
\sup_{z\in\mathbb S}\sup_{X':\,X'_0=x,\,X'_s=z}
\mathbf\Pi_x(X')\sup_{\theta_sX:\,(\theta_sX)_t=y}\mathbf\Pi_z(\theta_sX)
=\sup_{z\in\mathbb S}\Pi_{x,s}(z)\Pi_{z,t}(y)\,.
\end{multline*}
By  \eqref{eq:26},
\begin{equation}
  \label{eq:71}
  \Pi_{x,s+t}(y)\ge \Pi_{x,s}(z)\Pi_{z,t}(y)\,.
\end{equation}
In addition, if $\Pi_{x,s+t}(y)=\mathbf\Pi_x(X)$\,, then
\begin{equation}
  \label{eq:70}
    \Pi_{x,s+t}(y)=\Pi_{x,s}(X_s)\Pi_{X_s,t}(y)\,.
\end{equation}
Let 
\begin{equation*}
  \tilde\Pi(x,y)=\sup_{t\ge 0}\Pi_{x,t}(y)\,.
\end{equation*}
It is noteworthy that
\begin{equation*}
  \tilde\Pi(x,y)\ge \tilde\Pi(x,z)\tilde\Pi(z,y)\,.
\end{equation*}
\begin{lemma}
  \label{le:attraction}
\begin{enumerate}
  \item   
For any $\eta>0$\,, $\delta>0$ and compact $K\subset\mathbb S$\,,
there exists $T>0$ such that,
if  $\mathbf\Pi_x(X)>\eta$\,,  then
\begin{equation*}
  \inf_{0\le t\le T}d(X_t,A)<\delta
\,,
\end{equation*}
where $x\in K$\,.
\item
Let $x^t\to x$ and $y^t\to y$\,, as $t\to\infty$\,. Then, there exists
the limit
\begin{equation*}
 \Pi(x,y)=\lim_{t\to\infty}\Pi_{x^t,t}(y^t)\,.
\end{equation*}
 In addition,
\begin{equation}
  \label{eq:68}
        \Pi(x,y)=\sup_{z\in\mathbb S}\Pi(x,z)\Pi(z,y)\,.
\end{equation}
\item The function $\Pi(x,y)$ is continuous in $(x,y)$ and is 
  upper compact in $y$ uniformly over compact sets of $x$\,. In
addition,
$\sup_{y\in\mathbb S}\Pi(x,y)=1$ making $\Pi(x,y)$ a deviability
density in $y$\,. If $A'\subset A$\,, then $\Phi_{\mathbf\Pi}(A',x)$ is a
continuous upper compact function of $x$\,.
\item Let $F$ be a closed subset of $\mathbb S$ with  nonempty
  boundary $\partial F$\,. If $x\notin F$\,, then there exists $\hat x\in\partial
  F$\,, such that $\Pi(x,\hat x)=\Pi(x,\partial F)=\Pi(x,F)$\,.
   If $\Pi(x,F)>0$\,, then
  $\tilde\Pi(\hat x, \tilde x)=1$\,, for any
  $\tilde x\in F$ with $\Pi(x,\tilde x)=\Pi(x,F)$\,.
\end{enumerate}\end{lemma}
\begin{proof}
For part 1, we  draw on the proof of part (a) of Lemma
2.2 in  Freidlin and Wentzell \cite[Chapter 4]{wf2}.
For $x'\in\mathbb S$\,, let $\Theta_{x'}$ denote the set of
trajectories $X'$ such that $\mathbf\Pi_{x'}(X')=1$\,. Owing to {\bf 2.2},
this set is compact.
  Let $O_\delta$ denote the open $\delta$--neighbourhood of $A$
and  let $T_{X'}=\inf\{s:\,X'_s\in O_\delta\}\le\infty$\,. Thanks to
{\bf 2.4}(1), $T_{X'}<\infty$ provided $X'\in\Theta_{x'}$\,.
Let
  $\tilde T_{x'}=\sup_{X'\in \Theta_{x'}}T_{X'}$\,. 
 It is a finite upper
  semicontinuous  function of $x'$ and the supremum is attained.
Indeed, let $T_{X'^n}\to \tilde T_{x'}$\,, where $X'^n\in
\Theta_{x'}$\,. By compactness of $\Theta_{x'}$\,, we may assume that
$X'^n\to X'\in\Theta_{x'}$\,. 
   Let $t'$ represent the length of time when
 $X'$  reaches the closed $\delta/2$--neighbourhood of $A$\,. Since
 $X'^n\to X'$ uniformly on $[0,t']$\,, for all $n$ great enough, the
 $X'^n$ reach $O_\delta$ by time $t'$\,,
 i.e., $T_{X'_n}\le t'$\,, proving that $\tilde
 T_{x'}<\infty$\,. 
As $X'^n_{T_{X'^n}-\gamma}\notin O_\delta$\,, 
for all $\gamma\in[0,T_{X'_n})$\,, we have that
$X'_{\tilde T_{x'}-\gamma}\notin O_\delta$\,, for $\gamma\in[0,\tilde T_{x'})$\,,
 hence, $\tilde T_{x'}-\gamma\le T_{X'}$\,, so, $\tilde T_{x'}\le T_{X'}$\,, proving
that $\tilde T_{x'}=T_{X'}$\,. To see that $\tilde T_{x'}$ is upper
semicontinuous,  let $x'^n\to x'$\,. There exist $X'^n\in
  \Theta_{x'^n}$ 
such that 
$\tilde T_{x'^n}=T_{X'^n}$\,. As $\mathbf\Pi_{x'^n}(X'^n)=1$\,, by {\bf 2.2},
the set of
the $X'^n$ is relatively compact, so, for a suitable subsequence,
$X'^n\to  X'\in \Theta_{x'}$\,, which implies by a similar
argument as before that $\limsup_{n\to\infty}T_{X'^n}\le T_{X'}$\,.

 Since, owing to {\bf 2.5},
$\limsup_{t\to\infty}\sup_{x\in K}
\mathbf\Pi_{x}(X':\,X'_t\not\in K')<\eta$\,, for suitable compact
$K'\supset K$\,, we may assume that $X_t\in K'$\,, for all $t\ge t_0$\,.
 Let $D$ represent the set of points that are at distances $\delta$ or
greater away from  $A$\,.
If $K'\cap D=\emptyset$\,, then  one can take $T=t_0+1$\,. We consider the
case where $K'\cap D\not=\emptyset$\,.
  The set  $K'\cap D$ being
compact, $\sup_{x'\in K'\cap D}\tilde T_{x'}$ is attained. 
We denote it by $T'$\,.
Since the set $\cup_{x'\in K'\cap D}\{X':\,
  \mathbf\Pi_{x'}(X')\ge\gamma\}$ is compact, for any $\gamma>0$\,,
$\sup_{x'\in K'\cap D}\sup_{X':\,X'_t\in K'\cap D\text{ for all
  }t\in[0,T']}\mathbf\Pi_{x'}(X')$ is attained.
 On the other hand, 
no $X'$ in the supremum is such that 
 $\mathbf\Pi_{x'}(X')=1$ because it takes it the length of
 time greater than $T'$ to get to
$O_\delta$\,. Therefore, 
 the supremum is less than 1. 
 We have that $(\theta_{t_0}X)_u\in K'\cap D$\,, for all  $u\ge 0$\,.
If $X_t$  belongs to $K'\cap D$\,, for all $t\in [t_0,t_0+N]$\,, then
\begin{equation*}
      \mathbf\Pi_{x}(X)\le\prod_{m=0}^N
\sup_{x'\in K'\cap D}\sup_{X'\in  \pi_{T'}^{-1}(\pi_{T'}(\theta_{m+t_0}X))}
\mathbf\Pi_{x'}(X')
\le \bl(\sup_{x'\in K'\cap D}\sup_{X'_t\in K'\cap D\text{ for all }
  t\in[0,T']} \mathbf\Pi_{x'}(X')\br)^N\,.
\end{equation*}
Hence, if $N$ is such that the rightmost side is less than $\eta$\,,
then
$X_t$ reaches $O_\delta$ by
time $t_0+N$\,. Such  $N$ depends on $\eta$\,, $\delta$ and $K$ only.

We prove part 2.
Let us begin by proving that
\begin{subequations}
  \begin{align}
  \label{eq:17}
  \limsup_{t\to\infty}\Pi_{x^t,t}(y^t)=\limsup_{t\to\infty}\Pi_{x,t}(y)
\intertext{ and }
  \liminf_{t\to\infty}\Pi_{x^t,t}(y^t)=\liminf_{t\to\infty}\Pi_{x,t}(y)\,.
  \label{eq:17a}
\end{align}
\end{subequations}
By {\bf 2.4}(5), given  $\epsilon>0$\,, there exist $t_0$ and $t_1$
such that $\Pi_{x^t,t_0}(x)\ge 1-\epsilon$
and $\Pi_{y,t_1}(y^t)\ge 1-\epsilon$\,, for all $t$ great enough.
Therefore, invoking \eqref{eq:71},
\begin{equation*}
  \Pi_{x^{t+t_0+t_1},t+t_0+t_1}(y^{t+t_0+t_1})\ge
  \Pi_{x^{t+t_0+t_1},t_0}(x)\Pi_{x,t}(y)
\Pi_{y,t_1}(y^{t+t_0+t_1})
\ge (1-\epsilon)^2\Pi_{x,t}(y)\,,
\end{equation*}
so,
\begin{equation*}
    \limsup_{t\to\infty}\Pi_{x^t,t}(y^t)\ge\limsup_{t\to\infty}\Pi_{x,t}(y)\,.
\end{equation*}
The other inequalities needed for \eqref{eq:17} and \eqref{eq:17a} are
proved similarly.

Let us prove that the $\Pi_{x,t}(y)$ converge. We may and will assume that
$\limsup_{t\to\infty}\Pi_{x,t}(y)>\eta>0$\,. 
Assuming $t$ is great enough,
let $X$ be such that $X_0=x$\,, $X_t=y$ 
and $\Pi_{x,t}(y)=\mathbf\Pi_{x}(X)\ge\eta$\,.
Given  $\epsilon>0$\,,
 part 1 and assumption {\bf 2.4}(4) furnish $\delta>0$ and $t>0$
 such that $X$ reaches the
closed $\delta$--neighbourhood of $A$ at some  $\tilde t\le t$ and,
for some  $s_0>0$\,,
 $s_1>0$\,,  and $a\in A$\,,
$\Pi_{X_{\tilde t},s_0}(a)\ge 1-\epsilon$
 and $\Pi_{a,s_1}(X_{\tilde t})\ge 1-\epsilon$\,. 
 Let $\tilde X$ be a trajectory such that 
$\tilde X_s=X_s$\,, for $s\le \tilde t$\,, 
$\tilde X_{\tilde t+s_0}=a$\,, $\mathbf\Pi_{X_{\tilde t}}(\theta_{\tilde t}\tilde
X)=\Pi_{X_{\tilde t},s_0}(a)$\,, 
$\tilde X_{\tilde t+s_0+s_1+s}=X_{\tilde t+s}$\,, for
$s\ge 0$\,, and
 $\mathbf\Pi_{a}(\theta_{\tilde t+s_0}\tilde
X)=\Pi_{a,s_1}(X_{\tilde t})$\,. 
(It is worthwhile keeping in mind that $\tilde t$\,,
$s_0$\,, $s_1$ and $a$ depend, generally speaking, on $x$\,.)
 In words,
$\tilde X$ agrees with $X$ until $\tilde t$\,, optimally 
gets from $X_{\tilde t}$ to $a$ in time $s_0$\,, 
 optimally 
comes back to $X_{\tilde t}$ 
in time $s_1$ and then retraces $X$ with a shift in
time. By {\bf 2.3}, 
\begin{multline}
  \label{eq:51}
      \mathbf\Pi_x(\tilde X)=\mathbf\Pi_x(\pi_{\tilde
        t}^{-1}(\pi_{\tilde t}
\tilde X))
\mathbf\Pi_{X_{\tilde t}}(\theta_{\tilde t}\tilde X)
=\mathbf\Pi_x(\pi_{\tilde t}^{-1}(\pi_{\tilde t} X))
\mathbf\Pi_{X_{\tilde t}}(\pi_{s_0}^{-1}
\pi_{s_0}(\theta_{\tilde t}\tilde X))
\mathbf\Pi_{a}(\theta_{\tilde t+s_0}\tilde X)\\
=\mathbf\Pi_x(\pi_{\tilde t}^{-1}(\pi_{\tilde t} X))
\mathbf\Pi_{X_{\tilde t}}(\pi_{s_0}^{-1}
\pi_{s_0}(\theta_{\tilde t}\tilde X))
\mathbf\Pi_{a}(\pi^{-1}_{s_1}\pi_{s_1}(\theta_{\tilde t+s_0}\tilde X))
\mathbf\Pi_{\tilde X_{\tilde t+  s_0+s_1}}(\theta_{\tilde t+
  s_0+s_1}\tilde X)
\\=\mathbf\Pi_x(\pi_{\tilde t}^{-1}(\pi_{\tilde t} X))
\mathbf\Pi_{X_{\tilde t}}(\pi_{s_0}^{-1}
\pi_{s_0}(\theta_{\tilde t}\tilde X))
\mathbf\Pi_{a}(\pi^{-1}_{s_1}\pi_{s_1}(\theta_{\tilde t+s_0}\tilde X))
\mathbf\Pi_{X_{\tilde t}}(\theta_{\tilde t} X)\\
=\mathbf\Pi_x(\pi_{\tilde t}^{-1}(\pi_{\tilde t} X))
\Pi_{X_{\tilde t},s_0}(a)
\Pi_{a,s_1}(X_{\tilde t})
\mathbf\Pi_{X_{\tilde t}}(\theta_{\tilde t} X)
\,.
 \end{multline}
Since $\Pi_{X_{\tilde t},s_0}(a)\ge 1-\epsilon$ and
$\Pi_{a,s_1}(X_{\tilde t})\ge 1-\epsilon$\,,
\begin{equation}
\label{eq:42}   \mathbf\Pi_x(\tilde X)\ge(1-\epsilon)^2
 \mathbf\Pi_x(\pi_{\tilde t}^{-1}(\pi_{\tilde t} X))
\mathbf\Pi_{X_{\tilde t}}(\theta_{\tilde t} X)= (1-\epsilon)^2\mathbf\Pi_x(X)
=(1-\epsilon)^2\Pi_{x,t}(y)\,. 
\end{equation}
Let $\hat{ \Pi}_{x,t}(y)$ denote the  deviability  of getting
from $x$ to $y$ in time $t$ with visiting $A$\,.
By \eqref{eq:42}, 
\begin{equation}
  \label{eq:57}
\hat{\Pi}_{x,t+s_0+s_1}(y)\ge\Pi_x(\tilde X)\ge
(1-\epsilon)^2\Pi_{x,t}(y)\,,\text{ for
  all }
x\in K\text{ and }y\in\mathbb S\,.
\end{equation}
Since staying
at a point of $A$ ''costs nothing'', $\hat{ \Pi}_{x,t}(y)$ is a monotonically
increasing function of $t$\,. More precisely, let $\hat X_0=x$\,,
$\hat X_t=y$\,, $\hat X_{\hat t}=\hat 
a\in A$\,, for some $\hat t\in[0,t]$\,,  and
$\mathbf\Pi_x(\hat X)=\hat{\Pi}_{x,t}(y)$\,. 
For $u>0$\,, we define $\hat X^u$ by $\pi_{\hat t}\hat X^u=\pi_{\hat
  t}\hat X$\,, $\hat X^u_s=\hat a$\,, for $s\in[\hat t,\hat t+u]$\,,
and $\theta_{\hat t+u}\hat X^u=\theta_{\hat t}\hat X$\,.
In analogy with \eqref{eq:51}, on noting that, by {\bf 2.4}(2),
$\mathbf \Pi_{\hat a,u}(\hat a)=1$, we have that
\begin{multline*}
\mathbf\Pi_x(\hat X^u)
  =\mathbf\Pi_x(\pi_{\hat t}^{-1}(\pi_{\hat t}\hat X))
\mathbf\Pi_{\hat X_{\hat t}}(\pi^{-1}_u\pi_u(\theta_{\hat t}\hat X^u))
\mathbf\Pi_{\hat X_{\hat t}}(\theta_{\hat t+u}\hat X^u)
=\mathbf\Pi_x(\pi_{\hat t}^{-1}(\pi_{\hat t}\hat X))
\mathbf\Pi_{\hat X_{\hat t}}(\theta_{\hat t}\hat X)\\
=\mathbf\Pi_x(\hat X)\,.
\end{multline*}
Therefore,  
$\hat{\Pi}_{x,t}(y)=\mathbf\Pi_x(\hat X)=\mathbf\Pi_x(\hat X^u)\le \hat{\Pi}_{x,t+u}(y)$\,. Hence, there exists $\lim_{t\to\infty}\hat{\Pi}_{x,t}(y)$\,.  By \eqref{eq:57}, 
$\lim_t\hat{\Pi}_{x,t}(y)\ge (1-\epsilon)^2\limsup_t\Pi_{x,t}(y)$\,.
On the other hand,
$\hat{\Pi}_{x,t}(y)\le\Pi_{x,t}(y)$ 
so that 
$\lim_t\Pi_{x,t}(y)$ exists and equals $\lim_t\hat{\Pi}_{x,t}(y)$\,.

We prove \eqref{eq:68}.
Taking limits in \eqref{eq:26} yields
\begin{equation*}
  \Pi(x,y)\ge\sup_{z\in\mathbb S}\Pi(x,z)\Pi(z,y)\,.
\end{equation*}
By \eqref{eq:26} and \eqref{eq:57}, for any $\epsilon>0$
 and  compact $K'$\,, for certain $s_0(x)$\,, $s_1(x)$\,, $s_0(z)$ and
 $s_1(z)$\,, assuming $s'$ and $t'$ are great enough,
\begin{equation*}
        \Pi_{x,s'+t'}(y)\le
\frac{1}{(1-\epsilon)^4}\sup_{z\in K'}\hat{\Pi}_{x,s'+s_0(x)+s_1(x)}(z)
\hat{\Pi}_{z,t'+s_0(z)+s_1(z)}(y)\vee \Pi_{x,s'}(\mathbb S\setminus K')
\vee\epsilon\,,
\end{equation*}
where  $u\vee v=\max(u,v)$\,.
Letting $s'\to\infty$ and $t'\to\infty$ yields, by monotonicity and {\bf
  2.5}, for suitable $K'$\,,
\begin{equation*}
          \Pi(x,y)\le
\frac{1}{(1-\epsilon)^4}\sup_{z\in K'}\Pi(x,z)
\Pi(z,y)\vee \epsilon
\,,
\end{equation*}
which concludes the proof.

In order to prove  that $\Pi(x,y)$ is continuous, 
we note that, by {\bf 2.4}(5)\,, given  $\epsilon>0$\,, there
exists $\delta>0$ such that, if $d(x,x')<\delta$ and
$d(y,y')<\delta$\,, then, for some $t_0$\,, $t_1$\,, $t_0'$ and
$t_1'$\,, we have that
 $\Pi_{x',t_0}(x)\ge 1-\epsilon$\,, $\Pi_{x,t_1}(x')\ge 1-\epsilon$\,,
$\Pi_{y',t_0'}(y)\ge 1-\epsilon$
and $\Pi_{y,t_1'}(y')\ge 1-\epsilon$\,.
Then $\Pi_{x',t+t_0+t_1'}(y')\ge
\Pi_{x',t_0}(x)\Pi_{x,t}(y)\Pi_{y,t_1'}(y')\ge (1-\epsilon)^2\Pi_{x,t}(y)$\,.
Similarly, $\Pi_{x,t+t_1+t_0'}(y)\ge
\Pi_{x,t_1}(x')\Pi_{x',t}(y')\Pi_{y',t_0'}(y)\ge
(1-\epsilon)^2\Pi_{x',t}(y')$\,.
Letting $t\to\infty$ obtains the result.
The upper compactness follows by the tightness property in {\bf 2.5}.
As $\sup_{y\in\mathbb S}\Pi_{x,t+s_0+s_1}(y)=1$\,, by \eqref{eq:57} and the
monotonic convergence $\hat{\Pi}_{x,t+s_0+s_1}(y)\uparrow \Pi(x,y)$\,, as $t\to\infty$\,,
$\sup_{y\in\mathbb S}\Pi(x,y)=1$\,. That $\Phi_\Pi(A',x)$ is
continuous and upper compact in $x$ follows from $A$ being locally finite.

We now prove part 4. We may and will assume that $\Pi(x,F)>0$\,.
 As $F$ is closed and $\Pi(x,y)$ is upper compact in $y$\,, there exists $\tilde x\in F$
such that $\Pi(x,\tilde x)=\Pi(x,F)$\,.
Consider $\hat X^t$\,, with $\hat X_0=x$ and $\hat X^t_t=\tilde x$\,, 
such that $\mathbf\Pi_x(\hat X^t)=\Pi_{x,t}(\tilde x)$\,. (Since
$\Pi_{x,t}(\tilde x)>0$\,, for all $t$ great enough,
the trajectory
$\hat X^t$ exists for all those $t$\,.)
Let $\hat t=\inf\{s:\, X^t_s\in F\}$ and
$\hat x^t=\hat X^t_{\hat t}$\,. Since 
$\hat X^t$ is a continuous trajectory,
$\hat x^t\in\partial F$\,.  
As $\hat X^t$ is an optimal trajectory from $x$ to $\tilde x$\,, as in
\eqref{eq:70}, 
\begin{equation}
  \label{eq:20}
  \Pi_{x,t}(\tilde x)=\Pi_{x,\hat t}(\hat x^t)
\Pi_{\hat x^t,t-\hat t}(\tilde x)\,.
\end{equation}
Let $t\to\infty$\,.  
 Suppose that, along a subsequence, $\hat t\to\infty$ and
  $t-\hat t\to\infty$\,. 
Since $\Pi_{x,t}(\tilde x)$ is bounded away from 0\,,
 for all $t$ great enough, so is $\Pi_{x,\hat t}(\hat
 x^t)$\,. By {\bf 2.5}, $\hat x^t$ is a relatively compact set of
 points of $\mathbb S$\,. Let $\hat x$ represents a limit point of the
 set. Letting $t\to\infty$\,, $\hat t\to\infty$ and $t-\hat
 t\to\infty$ in \eqref{eq:20} yields, by Lemma \ref{le:attraction},
$\Pi(x,\tilde x)=\Pi(x,\hat x)\Pi(\hat
x,\tilde x)$\,. As $\hat x\in F$ and $\Pi(x,\tilde x)\ge
\Pi(x,y)$\,, for all $y\in  F$\,,
 we have that $\Pi(\hat x,\tilde x)=1$ and $\Pi(x,\hat
 x)=\Pi(x,\tilde x)$\,. 

Let us consider the case where, for some subsequence,
$\hat t\to\infty$ and $t-\hat t$ is bounded.
Since $\Pi(x,F)>0$\,, we have that $\lim_{t\to\infty} \mathbf\Pi_x(\hat X^t)>0$\,, so, 
along a subsubsequence, $\hat X^t\to \hat X$ and $t-\hat t\to \breve
t$\,. Hence, by {\bf 2.2},
$\limsup\mathbf\Pi_{\hat x^t}(\theta_{t-\hat t}\hat X^t)\le \mathbf\Pi_{\hat
  x}(\theta_{\breve t}\hat X)\le\tilde\Pi(\hat x,\tilde x)$\,. By \eqref{eq:20} once again, on noting that
$\Pi_{\hat x^t,t-\hat t}(\tilde x)=\mathbf\Pi_{\hat x^t}(\theta_{t-\hat
  t}\hat X)$\,, we have that
$\Pi(x,\tilde x)\le\Pi(x,\hat x)\tilde\Pi(\hat
x,\tilde x)$\,, which enables us to obtain, in a similar fashion as
before, that $\tilde\Pi(\hat x,\tilde x)=1$ and that $\Pi(x,\hat
x)=\Pi(x,\tilde x)$\,.
The case where, along a subsequence, $\hat t $ is bounded and $t-\hat
t\to\infty$\,, is tackled similarly.
\end{proof}
\begin{remark}
  As the proof of part 4 shows, if $\Pi(x,F)>0$\,, then
  $\partial F\not=\emptyset$\,.
\end{remark}
By \eqref{eq:68}, for all sets $\Gamma$ and $\Gamma'$\,,
\begin{equation}
  \label{eq:33}
  \Pi(\Gamma,\Gamma')=\sup_{C\subset\mathbb S}\Pi(\Gamma,C)\Pi(C,\Gamma')\,.
\end{equation}
The following implication is noteworthy:
\begin{equation}
  \label{eq:38}
  \Pi(\Gamma,C)\Pi(C,\Gamma')\le \Pi(\Gamma,\Gamma')\,.
\end{equation}
We also note that $\Pi(x,y)\le \tilde\Pi(x,y)$ and that 
$\Pi(x,a)=\tilde\Pi(x,a)$\,, provided $a\in A$\,.
\begin{lemma}
  \label{le:rep_1}
The function $\Pi(x)$ is continuous,
\begin{equation}
  \label{eq:9}
  \Pi(x,y)=\sup_{a\in A }\Pi(x,a)\Pi(a,y)
\end{equation}
and \begin{equation}
  \label{eq:1a}
\mathbf  \Pi(x)=\sup_{ a\in A }\mathbf \Pi(a)\Pi(a,x)\,.
\end{equation}
In addition, $\Pi(x,A)=1$\,, for all $x$\,, and $\mathbf\Pi(A)=1$\,.
 \end{lemma}
\begin{proof}
We show that
  \begin{equation}
    \label{eq:21}
\mathbf    \Pi(x)=\sup_{y\in\mathbb S}\mathbf\Pi(y)\Pi(y,x)\,.
  \end{equation}
By $P^n$ being invariant, for a nonnegative bounded continuous
function $f$\,,
\begin{equation*}
\int_{\mathbb S}f(x)^n\,P^n(dx)=
  \int_{\mathbb S}\int_{\mathbb S}f(y)^n\,\mathbf P(X^n_t\in dy|X^n_0=x)\,P^n(dx)\,.
\end{equation*}
Since $\bl(\int_{\mathbb S}f(y)^n\,\mathbf P(X^n_t\in
dy|X^n_0=x_n)\br)^{1/n}\to \sup_{y\in\mathbb S}f(y)\Pi_{x,t}(y)$
provided $x_n\to x$\,, with $x_n$ coming from the support of the
distribution of $P^n$\,,  and since the
$P^n$ LD converge to $\mathbf\Pi$ along a subsequence, 
by Lemma \ref{le:cont},
\begin{equation}
  \label{eq:11}
  \sup_{x\in\mathbb S}f(x)\, \mathbf\Pi(x)=
  \sup_{x\in\mathbb S}\sup_{y\in\mathbb S}f(y)\, \Pi_{x,t}(y) \mathbf\Pi(x)\,.
\end{equation}
   $\sup_{x\in\mathbb S}
\Pi_{x,t}(y)\mathbf\Pi(x)$ may be taken over  $x$ from a compact set,
so,  it  is an upper semicontinuous function of $y$\,. By \eqref{eq:11},
\begin{equation*}
  \mathbf \Pi(y)=\sup_{x\in\mathbb S}
\Pi_{x,t}(y)\mathbf\Pi(x)\,,
\end{equation*}
 which implies \eqref{eq:21} in analogy with the proof of \eqref{eq:68}.
By a similar argument as above, the righthand side of \eqref{eq:21}
   is an upper
semicontinuous function of $x$\,. On the other hand, it is a lower
semicontinuous function of $x$ because  $\Pi(x,y)$ is a continuous
function of $y$\,.
Hence, $\mathbf\Pi(x)$ is continuous.

By \eqref{eq:33},
 $\Pi(x,y)\ge\Pi(x,a)\Pi(a,y)$\,, for all $a\in A$\,, so, when proving 
\eqref{eq:9}, we may assume that $\Pi(x,y)>0$\,.
By the proof of part 2 of Lemma \ref{le:attraction}, given
$\epsilon>0$\,, for all $t$ great enough
$\Pi_{x,t}(y)\le \hat{\Pi}_{x,t}(y)/(1-\epsilon)$\,, where $\hat
{\Pi}_{x,t}(y)$ is the  deviability to get from $x$ to $y$ in
time $t$ with visiting $A$\,. We may also assume that
$\Pi_{x,t}(y)\ge (1-\epsilon)\Pi(x,y)$\,. Let $\hat X$ be such that 
$\hat X_0=x$\,, $\hat X_t=y$\,, $\hat X_{\hat t}=\hat a$\,, where
$\hat a\in A$\,, $\hat
t\in [0,t]$\,, and $\mathbf\Pi_x(\hat X)\ge (1-\epsilon)^2\Pi_{x,t}(y)$\,.
We have that $\Pi_{x,\hat t}(\hat a)\ge \mathbf\Pi_x(\pi_{\hat t}^{-1}(\pi_{\hat
  t}\hat X))$ and $\Pi_{\hat a,t-\hat t}(y)\ge \mathbf\Pi_{\hat a}(\theta_{\hat t}\hat
X)$ so that  
\begin{equation*}
  \Pi_{x,\hat t}(\hat a)\Pi_{\hat a,t-\hat t}(y)
\ge \mathbf\Pi_x(\pi_{\hat t}^{-1}\pi_{\hat
  t}(\hat X)) \mathbf\Pi_{\hat a}(\theta_{\hat t}\hat
X)=\mathbf\Pi_x(\hat X)\ge(1-\epsilon)^3\Pi(x,y)\,.
\end{equation*}
As follows from the proof of part 2 of Lemma \ref{le:attraction},
$\Pi(x,\hat a)\ge\Pi_{x,\hat t}(\hat a)$ and $\Pi(\hat a,y)
\ge\Pi_{\hat a,t-\hat t}(y)$ so that
$\Pi(x,\hat a)\Pi(\hat a,y)\ge (1-\epsilon)^3\Pi(x,y)$\,.

For \eqref{eq:1a}, note that, by \eqref{eq:21} and \eqref{eq:9},
\begin{multline*}
  \mathbf \Pi(x)=\sup_{y\in\mathbb S}\mathbf\Pi(y)\Pi(y,x)=
\sup_{y\in\mathbb S}\mathbf\Pi(y)\sup_{a\in A}(\Pi(y,a)\Pi(a,x))\\
=\sup_{a\in A}\sup_{y\in\mathbb S}(\mathbf\Pi(y)\Pi(y,a))\Pi(a,x)
=\sup_{a\in A}\mathbf\Pi(a)\Pi(a,x)\,.
\end{multline*}
By  \eqref{eq:9} and part 3 of Lemma \ref{le:attraction},
$1=\sup_{y\in \mathbb S}\Pi(x,y)= \sup_{a\in A}\Pi(x,a)$\,,
so, $  \Pi(x,A)=1\,.$ The argument for $\mathbf\Pi(A)=1$ is analogous.
 \end{proof}
 \begin{remark}
   One can prove that if 
given  $a\in A$\,, $x\not=a$\,, we have that
  $\Pi_{a,t}(x)<1$\,, for   any $t>0$\,, then
 $\Pi(a,x)<1$\,. 
 \end{remark}
\begin{remark}
  \label{re:norm}
If $\mathbf \Pi(x)\ge\epsilon>0$\,, then the $\sup$ in
\eqref{eq:1a} may be taken over a finite collection of $a$ which
depends on $\epsilon$ only. It follows by the fact that the
collection of $a$ with $\Pi(a)\ge\epsilon$ is finite.
As a consequence, the supremum is attained at some $a$\,. 
It is similar with $\mathbf\Pi(A)=1$ and $\Pi(x,A)=1$\,.
\end{remark}
\begin{remark}
  \label{re:put'}
By {\bf 2.4}(3), for $a,a'\in A$\,, $\Pi(a,a')>0$\,, whence
$\mathbf\Pi(a)>0$\,, for all $a\in A$\,.
\end{remark}
\begin{remark}
  \label{re:cont}
 $\Pi(x,y)$ and  $\mathbf \Pi(x)$ being continuous imply that
\begin{equation*}
  \Phi_{\mathbf \Pi}(\Gamma,\Gamma')=\Phi_{\mathbf \Pi}( \Gamma,\cl \Gamma')=
\Phi_{\mathbf \Pi}(\cl \Gamma,\Gamma')=\Phi_{\mathbf \Pi}(\cl \Gamma,\cl \Gamma')\,.
\end{equation*}
\end{remark}
\begin{lemma}[Max balance I]
    \label{the:maxflow}
 Let $\Gamma$ represent a Borel subset of $\mathbb S$\,.
Then,
 \begin{equation}
   \label{eq:10}
\Phi_{\mathbf \Pi}(\nt \Gamma,\nt \Gamma^c)\le\Phi_{\mathbf \Pi}(\cl \Gamma^c,\cl \Gamma)\,.
\end{equation}
\end{lemma}
\begin{proof}
  By \eqref{eq:1},
\begin{equation*}
  P^n(\Gamma)=\int_{\mathbb S} \mathbf P(X^n_t\in \Gamma|X^n_0=x)P^n(dx)\,.
\end{equation*}
It follows that
\begin{multline*}
    \int_\Gamma\Bl( \mathbf P(X^n_t\in \Gamma|X^n_0=x)
+\mathbf P(X^n_t\in \Gamma^c|X^n_0=x)\Br)\,P^n(dx)\\=
    \int_\Gamma\mathbf P(X^n_t\in \Gamma|X^n_0=x)\,P^n(dx)
+\int_{\Gamma^c}\mathbf P(X^n_t\in \Gamma|X^n_0=x)\,P^n(dx)\,,
\end{multline*}
so, on cancelling like terms,
\begin{equation}
  \label{eq:16}
        \int_\Gamma\mathbf P(X^n_t\in \Gamma^c|X^n_0=x)\,P^n(dx)=
\int_{\Gamma^c}\mathbf P(X^n_t\in \Gamma|X^n_0=x)\,P^n(dx)\,.
\end{equation}
By {\bf 2.1}, if $x^n\to
x$\,, with $x^n$ coming from the support of $P^n$\,, then
  \begin{align*}
\liminf_{n\to\infty}  \mathbf P(X^n_t\in \Gamma^c|X^n_0=x^n)^{1/n}
  \ge \sup_{y\in \nt \Gamma^c}\Pi_{x,t}(y)
\intertext{and}\limsup_{n\to\infty}\mathbf P(X^n_t\in \Gamma|X^n_0=x^n)^{1/n}
\le \sup_{y\in \cl \Gamma}\Pi_{x,t}(y)\,.
\end{align*}
By the LD convergence of the $P^n$ to $\Pi$ along  subsequence $n'$
and Lemma \ref{le:cont}\,,
\begin{align*}\liminf_{n'\to\infty}
\Bl(\int_\Gamma\mathbf P(X^{n'}_t\in \Gamma^c|X^{n'}_0=x)\,P^{n'}(dx)\Br)^{1/n'}
\ge \sup_{x\in\nt \Gamma}\sup_{y\in\nt \Gamma^c}\Pi_{x,t}(y)\,\mathbf \Pi(x)
\intertext{and}\limsup_{n'\to\infty}
\Bl(\int_{\Gamma^c}\mathbf P(X^{n'}_t\in  \Gamma|X^{n'}_0=x)\,P^{n'}(dx)\Br)^{1/n'}
\le \sup_{x\in\cl \Gamma^c}\sup_{y\in \cl \Gamma}\Pi_{x,t}(y)\,\mathbf \Pi(x)\,.
\end{align*}
Hence, by \eqref{eq:16},
\begin{equation*}
  \sup_{x\in\nt \Gamma}\sup_{y\in\nt  \Gamma^c}\Pi_{x,t}(y)\,\mathbf \Pi(x)\le
\sup_{x\in \cl \Gamma^c}\sup_{y\in \cl \Gamma}\Pi_{x,t}(y)\,\mathbf \Pi(x)\,.
\end{equation*}
 Letting $t\to\infty$ yields \eqref{eq:10}, as in the proof of part 2
 of Lemma \ref{le:attraction}.
\end{proof}
Let $\{A',A''\}$ be a nontrivial partition of $A$\,, i.e.,
$A'\cup A''=A$\,, $A'\cap A''=\emptyset$\,, $A'\not=\emptyset$ and
$A''\not=\emptyset$\,. As above, we let  $\partial $ stand for  the boundary of a  set.
\begin{lemma}
  \label{le:attractor}
 Let $\Gamma$ represent a  subset of
$\mathbb S$ such that 
$A'\subset\nt \Gamma$\,,
$A''\subset\nt \Gamma^c$ 
and 
$\Pi(a,\nt \Gamma)=\Pi(a,\cl \Gamma)$\,, for all
$a\in A''$ that satisfy
$\Pi(a)\ge \Phi_{\mathbf \Pi}( A'',\nt \Gamma)$\,.

If 
\begin{equation}
  \label{eq:27}
  \Phi_{\mathbf \Pi}( A'',\nt \Gamma)>
\Phi_{\mathbf \Pi}( A',\cl \Gamma^c)\,,
\end{equation}
then there exist $\hat x\in\partial \Gamma$\,, $\tilde a\in A''$
 and $\hat a\in A''$ such that 
$\Pi(\hat
x,\tilde a)=1$\,, 
 $\Pi(\hat a,\hat x)=\Pi(\hat a,\cl \Gamma)$ and 
$\mathbf\Pi(\hat x)=\mathbf
\Pi(\hat a)\Pi(\hat a,\hat x)=
\Phi_{\mathbf \Pi}( A'',\nt
\Gamma)=\Phi_{\mathbf \Pi}(\nt \Gamma^c,\nt \Gamma)=
\Phi_{\mathbf \Pi}(\cl \Gamma,\cl \Gamma^c)
$\,. Also, $\mathbf\Pi(\hat x)=\mathbf\Pi(\partial \Gamma)$\,. 

\end{lemma}
\begin{proof}
  We apply Lemma \ref{the:maxflow}.
Since $\cl \Gamma^c\supset
A''$
and $\cl \Gamma\supset
A'$\,, 
with the use of Lemma \ref{le:rep_1} and on recalling
 that $\Pi(a,a)=1$\,,
for $a\in A$\,,  as  well as \eqref{eq:38}, we can write
 the max flux $\cl \Gamma\to\cl \Gamma^c$ 
  as follows, 
  \begin{multline*}
    \Phi_{\mathbf \Pi}(\cl \Gamma,\cl \Gamma^c)=\sup_{x\in \cl \Gamma }\mathbf \Pi(x)\Pi(x,\cl \Gamma^c)=
    \sup_{x\in \cl \Gamma}\mathbf \Pi(x)\sup_{a\in A}\Pi(x,a)\Pi(a, \cl \Gamma^c)\\
=
\sup_{x\in \cl \Gamma}\bl(
\sup_{a\in A}\mathbf\Pi(a)\Pi(a,x)\br)\bl(\sup_{a\in A'}
\Pi(x,a)\Pi(a, \cl \Gamma^c)\vee 
\sup_{a\in A''}\Pi(x,a)\br)\\
=
\sup_{a\in A}\mathbf\Pi(a)\bl(\sup_{a'\in A'}\sup_{x\in \cl \Gamma}
\Pi(a,x)
\Pi(x,a')\Pi(a', \cl \Gamma^c)\vee 
\sup_{a''\in A''}\sup_{x\in \cl \Gamma}
\Pi(a,x)\Pi(x,a'')\br)\\
=
\sup_{a\in A'}\mathbf\Pi(a)\sup_{a'\in A'}\sup_{x\in
  \cl \Gamma}
\Pi(a,x)\Pi(x,a')\Pi(a', \cl \Gamma^c)\vee 
\sup_{a\in A'}\mathbf\Pi(a)\sup_{a''\in A''}\sup_{x\in
  \cl \Gamma}
\Pi(a,x)\Pi(x,a'')\\\vee
\sup_{a\in A''}\mathbf\Pi(a)\sup_{a'\in A'}\sup_{x\in
  \cl \Gamma}
\Pi(a,x)\Pi(x,a')\Pi(a', \cl \Gamma^c)\vee
\sup_{a\in A''}
\mathbf\Pi(a)\sup_{a''\in A''}\sup_{x\in
  \cl \Gamma}
\Pi(a,x)\Pi(x,a'')
\\=
\sup_{a\in A'}\mathbf\Pi(a)
\Pi(a, \cl \Gamma^c)\vee 
\sup_{a\in A'}\mathbf\Pi(a)
\Pi(a,A'')\vee
\sup_{a \in A''}\mathbf\Pi(a)\sup_{a'\in A'}
\Pi(a,a')\Pi(a', \cl \Gamma^c)\\\hspace{8cm}\vee
\sup_{a\in A''}
\mathbf\Pi(a)\sup_{a''\in A''}\sup_{x\in
  \cl \Gamma}
\Pi(a,x)\Pi(x,a'')
\\=
\Phi_{\mathbf \Pi}(A',\cl \Gamma^c)\vee 
\sup_{a\in A''}\mathbf\Pi(a)
\bl(\sup_{a'\in A'}
\Pi(a,a')\Pi(a',\cl \Gamma^c)\vee
\sup_{x\in \cl \Gamma}\Pi(a,x)\Pi(x,A'')\br)\,.
  \end{multline*}
Let us note that
\begin{equation*}
  \sup_{a\in A''}\mathbf\Pi(a)
\sup_{a'\in A'}
\Pi(a,a')\Pi(a',\cl \Gamma^c)\le 
\sup_{a'\in A'}
\Pi(a')\Pi(a',\cl \Gamma^c)= \Phi_{\mathbf \Pi}(A',\cl \Gamma^c)\,,
\end{equation*}
which yields
\begin{equation}
  \label{eq:28}
  \Phi_{\mathbf \Pi}(\cl \Gamma,\cl \Gamma^c)=
\Phi_{\mathbf \Pi}(A', \cl \Gamma^c)\vee 
\sup_{x\in \cl \Gamma}\Phi_{\mathbf \Pi}(A'',x)
\Pi(x,A'')\,.
\end{equation}
Similarly, owing to the conditions that
$A'\subset\nt \Gamma$ and
$ A''\subset\nt \Gamma^c$\,,
 the max flux from $\nt \Gamma^c$ to $\nt \Gamma$ is written as 
\begin{equation*}
  \Phi_{\mathbf \Pi}(\nt \Gamma^c,\nt \Gamma)=
\Phi_{\mathbf \Pi}( A'', \nt \Gamma)\vee 
\sup_{x\in \nt \Gamma^c}\Phi_{\mathbf \Pi}(A',x)\Pi(x,A')\,.
\end{equation*}

Since
$\Phi_{\mathbf \Pi}(A', \nt \Gamma^c)<\Phi_{\mathbf \Pi}(A'', \nt \Gamma)$\,,
we have that $\Phi_{\mathbf \Pi}(\nt \Gamma^c,\nt \Gamma)=\Phi_{\mathbf \Pi}( A'',\nt \Gamma)$\,. 
By max balance I (Lemma \ref{the:maxflow}), $\Phi_{\mathbf \Pi}(\cl \Gamma,\cl \Gamma^c)\ge\Phi_{\mathbf \Pi}(\nt \Gamma^c,\nt
\Gamma)$\,, so, by 
\eqref{eq:27} and \eqref{eq:28},
\begin{equation}
  \label{eq:30}
\sup_{a\in A''}\mathbf\Pi(a)  \sup_{x\in \cl \Gamma}\Pi(a,x)
\Pi(x,A'')\ge\Phi_{\mathbf \Pi}( A'',\nt \Gamma)\,.
\end{equation}
By \eqref{eq:27},
the latter righthand side is positive. Hence, the  $a$ on the
lefthand side can be assumed to belong to a compact set.  Since $A$ is locally
finite, the supremum can actually be taken over a finite set of $a$\,, so, it is
attained at some $\hat a$\,. 
By \eqref{eq:30},  $\mathbf\Pi(\hat a)\Pi(\hat
a,\cl \Gamma)\ge\Phi_{\mathbf \Pi}(A'',\nt \Gamma)$\,,
 so, by hypotheses,
$\Pi(\hat a,\nt \Gamma)=\Pi(\hat a,\cl \Gamma)$\,. 
Therefore, by using \eqref{eq:30} once again, 
\begin{equation}
  \label{eq:75}
  \sup_{x\in
 \cl \Gamma}\Pi(\hat a,x)\Pi(x,A'') \ge
\Pi(\hat a,\cl \Gamma)\,.
\end{equation}
By $\cl \Gamma$ being closed, there exists
$\tilde
 x\in \cl \Gamma$ such
that $\Pi(\hat a,\tilde x)=\Pi(\hat a,\cl \Gamma)$\,. Furthermore,
by $\Pi(\hat a,\cl \Gamma)$ being positive, \eqref{eq:75} implies that $\tilde
x$ can be chosen such that 
$\Pi(\tilde x,A'')=1$\,. By  $A$ being locally finite, there
exists $\tilde a\in A''$ with $\Pi(\tilde x,\tilde a)=1$\,.
  We also have that equality holds in
\eqref{eq:30}, so,
\begin{equation}
  \label{eq:69}
  \Phi_{\mathbf \Pi}(\cl \Gamma,\cl \Gamma^c)=\Phi_{\mathbf \Pi}(\nt \Gamma^c,\nt \Gamma)=
\mathbf\Pi(\hat a)\Pi(\hat a,\cl \Gamma)=\mathbf\Pi(\hat a)\Pi(\hat a,\nt \Gamma)=
\mathbf\Pi(\hat a)\Pi(\hat a,\tilde x)\,.
\end{equation}

By part 4 of Lemma \ref{le:attraction}, there exists $\hat
x\in\partial \Gamma$ such that
 $\tilde\Pi(\hat x,\tilde x)=1$ and $\Pi(\hat a,\hat
 x)=\Pi(\hat a,\tilde x)$\,. 
Therefore, $\Pi(\hat x,\tilde a)=\tilde\Pi(\hat x,\tilde a)\ge 
\tilde\Pi(\hat x,\tilde x)\tilde\Pi(\tilde x,\tilde a)=
\tilde\Pi(\tilde x,\tilde a)=\Pi(\tilde x,\tilde a)$\,, so, $\Pi(\hat
x,\tilde a)=1$\,.

To conclude, note that \eqref{eq:69} yields
$  \mathbf\Pi(\hat x)=\Phi_{\mathbf \Pi}(A,\hat x)
=\mathbf\Pi(\hat a)\Pi(\hat a,\hat x)\,.
$ If  $x\in\partial \Gamma$\,, then  $\Phi_{\mathbf \Pi}(\cl \Gamma,\cl \Gamma^c)\ge
\mathbf \Pi(x)$\,,
which implies that $\mathbf \Pi(x)\le \mathbf\Pi(\hat x)$\,.
\end{proof}
\begin{remark}
  If
$
    \Phi_{\mathbf \Pi}( A'',\nt \Gamma)=
\Phi_{\mathbf \Pi}( A',\cl \Gamma^c)\,$,
 then both expressions equal $\Phi_{\mathbf \Pi}(\cl \Gamma,\cl \Gamma^c)=\Phi_{\mathbf \Pi}(\nt \Gamma^c,\nt \Gamma)$\,.
\end{remark}
Let
  \begin{align}
&    F=\{x:\,\Phi_{\mathbf \Pi}( A',x)\ge
\Phi_{\mathbf \Pi}( A'',x)\},\notag
\\\notag
&\tilde A'=\{a'\in A':\,\mathbf\Pi(a')>\mathbf\Pi(a'')\Pi(a'',a')\,, \text{ for all
}a''\in A''\},\\
 & \tilde A''=\{a''\in A'':\, \mathbf\Pi(a'')>\mathbf\Pi(a')\Pi(a',a'')\,,\text{ for all
  }a'\in A'\},\notag  
\intertext{ and }
  &\breve A'=A'\cup (A''\setminus \tilde A'')
\,.\notag
\end{align}
 By \eqref{eq:1a}, 
 \begin{equation}
   \label{eq:12}
\breve A'\subset  F\,,
 \end{equation}
so $ F\not=\emptyset$\,.
As $\mathbf \Pi(x)$ and  $\Pi(x,y)$ are continuous, 
$F$  is a closed set which agrees with the
closure of its interior.
One can see that  $F^c\not=\emptyset$ if and only if 
 $\tilde A''\not=\emptyset\,$.
\begin{lemma}
  \label{le:ravno}
  \begin{enumerate}
  \item 
Suppose that $\tilde A''\not=\emptyset$\,. Then, there exist $\hat x\in\partial F$\,,
$\tilde a\in\tilde A''$ 
and $\hat a
\in\tilde A''$ such that
$\Pi(\hat x,\tilde a)=1$\,,
$\Pi(\hat a,\hat x)=\Pi(\hat a,
F)$
and $\Phi_{\mathbf \Pi}(
F,
F^c)=\Phi_{\mathbf \Pi}(A',F^c)=\Phi_{\mathbf \Pi}(A',\hat x)
=\Phi_{\mathbf \Pi}(
F^c,
F)=\Phi_{\mathbf \Pi}(A'',F)=\Phi_{\mathbf \Pi}(A'',\hat x)=
\mathbf\Pi(\hat a)\Pi(\hat a,\hat x)=\mathbf\Pi(\hat x)$\,.
\item If $\tilde A''=\emptyset$\,, then
$\Phi_{\mathbf \Pi}(A'',A')=\Phi_{\mathbf \Pi}(A',A'')=\mathbf\Pi(A'')$\,.
  \end{enumerate}\end{lemma}
\begin{proof}
The definitions of $\tilde A''$ and $  \breve A'$ imply that
\begin{equation}
  \label{eq:74}
  \mathbf\Pi(a'')>\mathbf\Pi(a')\Pi(a',a'')\text{ for all }a''\in\tilde A''\,, a'\in
  \breve A'\,.
\end{equation}
(Note that if $a'\in A''\setminus\tilde A''$\,, then $\mathbf\Pi(a')\le
\mathbf\Pi(\tilde a')\Pi(\tilde a',a')$, 
 for some $\tilde a'\in A'$\,, so that 
$\mathbf\Pi(a')\Pi(a',a'')\le
\mathbf\Pi(\tilde a')\Pi(\tilde a',a')\Pi(a',a'')
\le \mathbf\Pi(\tilde a')\Pi(\tilde a',a'')<\mathbf \Pi(a'')$\,.) 

Since 
$\Phi_{\mathbf \Pi}(A''\setminus\tilde A'',x)\le \Phi_{\mathbf \Pi}(A',x)\,,
$ for all $x$\,, we have that
$     F=\{x:\,\Phi_{\mathbf \Pi}(\breve A',x)\ge
\Phi_{\mathbf \Pi}(\tilde A'',x)\}\,.$
 By \eqref{eq:74},
 \begin{equation}
   \label{eq:22}
   \tilde A''\subset  F^c\,.
 \end{equation}

Let us show that $F\cap \cl \tilde A''=\emptyset$\,. 
 Suppose the opposite. Then, there exists sequence $a''_k\in \tilde A''$ that
converges to $x\in F$\,, as $k\to\infty$\,. Since the $a''_k$ form a
relative compact, they all must agree with $x$\,, for $k$ great enough, by the
local finiteness of $A$\,. Hence, $x\in \tilde A''$\,, so, 
 $x\notin F$\,.

Let  $\Gamma$ and $\Gamma'$ be disjoint  open
 neighbourhoods of $F$ and $\tilde A''$\,, respectively.
We have that
 $\breve A'\subset F\subset \Gamma=\nt \Gamma$ and
$\tilde  A''\subset \Gamma'\subset\nt \Gamma^c$\,.
Owing to $\Pi(x,y)$ being continuous and $\Gamma$ being open,
$\Pi(a,\nt\Gamma)=\Pi(a,\Gamma)=\Pi(a,\cl \Gamma)$\,, for all $a\in \tilde A''$\,.
Also, with the use of Lemma \ref{le:attraction}\,, 
$  \Phi_{\mathbf \Pi}(\tilde A'',\Gamma)=\Phi_{\mathbf \Pi}(\tilde A'',\partial \Gamma)>
\Phi_{\mathbf \Pi}(\breve A',\partial \Gamma)=
\Phi_{\mathbf \Pi}(\breve A', \Gamma^c)\,,$ where the inequality holds because
$\partial \Gamma\subset F^c$\,.
Thus, the set $\Gamma$ satisfies the hypotheses of Lemma
\ref{le:attractor}, with $\breve A'$ as $A'$ and $\tilde A''$ as $A''$\,.


Thanks to Lemma \ref{le:attractor}, there exist 
$\hat
x_\Gamma\in \partial \Gamma$ and $\hat a_\Gamma\in  \tilde A''$
such that $\Pi(\hat
x_\Gamma, \tilde A'')=1$\,,
$\Pi(\hat a_\Gamma,\hat x_\Gamma)=
\Pi(\hat a_\Gamma, \Gamma)$
and 
\begin{multline}
  \label{eq:39}
\mathbf\Pi(\hat x_\Gamma)=\mathbf\Pi(\hat a_\Gamma)\Pi(\hat
a_\Gamma,\hat x_\Gamma)=  \Phi_{\mathbf \Pi}(\cl \Gamma, \Gamma^c)=
\Phi_{\mathbf \Pi}(\nt \Gamma^c,  \Gamma)
=\Phi_{\mathbf \Pi}( \tilde A'', \Gamma)\\\ge
\Phi_{\mathbf \Pi}( \tilde A'',\breve A')>0\,,
\end{multline}
the latter inequality using that $\breve A'\not=\emptyset$\,, that
$\tilde A''\not=\emptyset$ and Remark \ref{re:put'}. 
The sets $\Gamma$ form a directed set by inclusion.
Accordingly, $\Gamma$\,, $\hat x_\Gamma$\,, $\hat a_\Gamma$ are nets.  
We look at their limits. The last inequality in \eqref{eq:39}
implies that the net
$\hat a_\Gamma$ contains only finitely many different points, so, we may
assume that there exists $\hat a\in  \tilde A''$ such that
$\Pi(\hat a,\hat x_\Gamma)=
\Pi(\hat a, \Gamma)$ and
\eqref{eq:39} holds with  $\hat a$ substituted for $\hat a_\Gamma$\,.
Also, the $\mathbf\Pi(\hat x_\Gamma)$ 
are bounded away from zero, which implies that
the $\hat x_\Gamma$ 
belong to a compact set, so, we may assume that $\hat x_\Gamma\to\hat
x\in \partial F$ and $\Pi(\hat a,\hat x_\Gamma)\to\Pi(\hat a,\hat x)$\,.
 By  \eqref{eq:39} and by $\mathbf \Pi(x)$ and  $\Pi(x,y)$ being
 continuous,
$\mathbf\Pi(\hat x)=\mathbf\Pi(\hat a)\Pi(\hat
a,\hat x)$\,. By the
convergence $\cl \Gamma\downarrow  F$\,,
$\Pi(\hat x, \tilde A'')=1$\,,
$\Pi(\hat a,\hat x)=
\Pi(\hat a, F)$
and $\mathbf\Pi(\hat a)\Pi(\hat a,\hat x)=\Phi_{\mathbf \Pi}
( \tilde A'', F)$\,.
Since $\Phi_{\mathbf \Pi}(\nt \Gamma^c, \Gamma)
\ge\Phi_{\mathbf \Pi}(\nt \Gamma^c, F)\uparrow \Phi_{\mathbf \Pi}( F^c, F)$
and $\Phi_{\mathbf \Pi}(\nt \Gamma^c, \Gamma)
\le \Phi_{\mathbf \Pi}( F^c, \Gamma)=\Phi_{\mathbf \Pi}( F^c,\cl \Gamma)\downarrow \Phi_{\mathbf \Pi}( F^c, F)$,
$\Phi_{\mathbf \Pi}(\nt \Gamma^c, \Gamma)\to
\Phi_{\mathbf \Pi}( F^c, F)$ so that 
$\Phi_{\mathbf \Pi}( F^c, F)=\mathbf\Pi(\hat x)$\,, see \eqref{eq:39}.
Since $\Phi_{\mathbf \Pi}( \tilde A'', \Gamma)=
\Phi_{\mathbf \Pi}( \tilde A'',\cl \Gamma)\downarrow 
\Phi_{\mathbf \Pi}( \tilde A'', F)$\,, by \eqref{eq:39} once again,
$\Phi_{\mathbf \Pi}( \tilde A'', F)=\Phi_{\mathbf \Pi}( F^c, F)\,.$
Similarly, $\Phi_{\mathbf \Pi}(\cl \Gamma, \Gamma^c)\ge
\Phi_{\mathbf \Pi}( F,(\cl \Gamma)^c)\uparrow \Phi_{\mathbf \Pi}( F,
F^c)$ and
$\Phi_{\mathbf \Pi}(\cl \Gamma, \Gamma^c)\le\Phi_{\mathbf \Pi}(
\cl \Gamma, F^c)\downarrow \Phi_{\mathbf \Pi}( F, F^c)$\,, so,
$\Phi_{\mathbf \Pi}(\cl \Gamma, \Gamma^c)\to\Phi_{\mathbf \Pi}( F,
F^c)=
\mathbf\Pi(\hat x)$\,.
Since $\hat x\in\partial F$\,, 
$\Phi_{\mathbf \Pi}( A',\hat x)=
\Phi_{\mathbf \Pi}( A'',\hat
x)$\,. 
Hence (recall \eqref{eq:1a} and that $A=A'\cup A''$), $\mathbf\Pi(\hat
x)=\Phi_{\mathbf \Pi}(  A,\hat x)=\Phi_{\mathbf \Pi}(  A',\hat x)=
\Phi_{\mathbf \Pi}(  A'',\hat x)$\,.

Let us consider the case that $\tilde A''=\emptyset$\,. 
It follows that $\mathbf\Pi(a'')=\sup_{a'\in A'}\mathbf\Pi(a')\Pi(a',a'')$\,,  for
all $a''\in A''$\,, so,
$\mathbf\Pi(A'')=\sup_{a'\in A'}\mathbf\Pi(a')\Pi(a',A'')=\Phi_{\mathbf \Pi}(A',A'')$\,.
 Suppose that $\tilde
A'\not=\emptyset$\,.
 Then, in analogy to part 1, with the roles of $A''$ and $A'$
switched and  $\tilde F=\{x:\,\Phi_{\mathbf \Pi}(A'',x)\ge \Phi_{\mathbf \Pi}(A',x)\}$\,,
for some $\tilde x\in\partial \tilde F$ such that
$\Pi(\tilde x,A')=1$\,,
$\Phi_{\mathbf \Pi}(A'',\tilde x)=\Phi_{\mathbf \Pi}(A'',\tilde
F^c)=\Phi_{\mathbf \Pi}(A',\tilde F)$\,. Since $\Phi_{\mathbf \Pi}(A'',\tilde
F^c)\le \mathbf\Pi(A'')$ and $\Phi_{\mathbf \Pi}(A',\tilde F)\ge\Phi_{\mathbf \Pi}(A',A'')=
\mathbf\Pi(A'')$
(note that $A''\subset\tilde F$), we obtain that
$\Phi_{\mathbf \Pi}(A'',\tilde x)=\Phi_{\mathbf \Pi}(A',\tilde x)=\mathbf\Pi(A'')$\,.
Since $\Phi_{\mathbf \Pi}(A'',\tilde x)=\sup_{a''\in A''}\mathbf\Pi(a'')\Pi(a'',\tilde
x)\Pi(\tilde x,A')\le
\sup_{a''\in A''}\mathbf\Pi(a'')\Pi(a'',A')=\Phi_{\mathbf \Pi}(A'',A')$\,, 
$\Phi_{\mathbf \Pi}(A'',A')=\mathbf\Pi(A'')$\,.

Finally, if $\tilde A''=\emptyset$ and
$\tilde A'=\emptyset$\,, then
it is elementary that 
$\Phi_{\mathbf \Pi}(A',A'')=\Phi_{\mathbf
  \Pi}(A'',A')=\mathbf\Pi(A'')=
\mathbf\Pi(A')$\,. \end{proof}
\begin{remark}
  The proof implies that when $\tilde A''=\emptyset$\,, $\Pi(A'',\tilde x)=1$ so that $\Pi(A'',A')=1$\,.
\end{remark}
\begin{remark}
  \label{re:ustoich}
One can prove that if 
given  $a\in A$ and $x\not=a$\,, we have that
  $\Pi_{a,t}(x)<1$\,, for  $t>0$\,, then
 $\Pi(a,y)<1$\,, for all $y\not=a$\,. It is seen to imply that
 $\tilde
 A''\not=\emptyset$ (and $\tilde A'\not=\emptyset$).
\end{remark}
The next lemma shows that
the  max fluxes  between
$  A'$
and $  A''$
balance, which condition uniquely
specifies the $\mathbf\Pi(a)\,, a\in A$\,.
\begin{lemma}[Max balance II]
  \label{the:max-bal-g}
The max balance in \eqref{eq:56} holds.
Those equations along with the normalisation condition that 
$\mathbf\Pi(A)=1$ uniquely specify  $\mathbf\Pi(a),\,a\in A $\,.
\end{lemma}
\begin{proof}
Let us suppose that $\tilde A''\not=\emptyset $ and
$\tilde A'\not=\emptyset $\,. 
  By Lemma \ref{le:ravno},
 it suffices to prove that
\begin{subequations}
  \begin{align}
  \label{eq:35}
\Phi_{\mathbf \Pi}( F, F^c)=\Phi_{\mathbf \Pi}(A', A'')\intertext{and
    that}
\Phi_{\mathbf \Pi}( F^c, F)=\Phi_{\mathbf \Pi}(A'', A')\,.\label{eq:35a}
\end{align}
\end{subequations}
Since $\tilde A''\not=\emptyset$\,, 
by part 1 of Lemma \ref{le:ravno}, using the notation of Lemma \ref{le:ravno},
\begin{equation*}
\Phi_{\mathbf \Pi}( F, F^c)=
\Phi_{\mathbf \Pi}(A',\hat x)=
\sup_{a\in A'}
\mathbf\Pi(a)\Pi(a,\hat x)\Pi(\hat x,\tilde a)
\le 
\sup_{a\in A'}
\mathbf\Pi(a)\Pi(a,\tilde a)
\le 
\Phi_{\mathbf \Pi}(A',\tilde A'')\,. \end{equation*}
On the other hand, as  $A'\subset F$ 
 and
$\tilde A''\subset F^c$ (see \eqref{eq:22}),
$\Phi_{\mathbf \Pi}(A',\tilde A'')
\le \Phi_{\mathbf \Pi}(F, F^c)$\,.
Relation \eqref{eq:35} has been proved, \eqref{eq:35a} is proved
similarly by switching the roles of $A''$ and $A'$ and using that
$\tilde A'\not=\emptyset$\,.
The case that $\tilde A''=\emptyset$ follows from part 2 of
 Lemma \ref{le:ravno}\,.
The case that $\tilde A'=\emptyset$ is dealt with analogously.

In order to prove the uniqueness of the restriction of $\mathbf\Pi$ to $A$\,,
let us suppose that,  for some deviability $\mathbf\Pi'$\,, 
the analogue of \eqref{eq:56} holds too.
Let $ A'=\{a\in A:\,\mathbf\Pi(a)>\mathbf\Pi'(a)\}$ and 
suppose that
$ A'\not=\emptyset$\,.
We have that $ A''=A\setminus A'\not=\emptyset$ because
$\mathbf\Pi'(A)=1$\,. 
Since
$\mathbf\Pi'(a)$ is small provided $a$ is outside of a certain compact and $A$
is locally finite,
$\sup_{a\in A'}\mathbf\Pi'(a)\Pi(a,A'')$ can be taken over a finite set, so, it is
attained. Hence, $
\Phi_{\mathbf \Pi}(A', A'')>\Phi_{\mathbf\Pi'}(A',
 A'')$\,. By the fact that $\mathbf\Pi'(a)\ge \mathbf\Pi(a)$ on $A''$ and  max balance,
$\Phi_{\mathbf\Pi'}(A'',A')\ge\Phi_{\mathbf \Pi}( A'',
 A')=\Phi_{\mathbf \Pi}(A', A'')\,,
$ so, $\Phi_{\mathbf \Pi}(A', A'')>\Phi_{\mathbf \Pi}(A', A'')$\,.
The contradiction proves that   $ A'=\emptyset$\,.
 \end{proof}
\begin{remark}
  The uniqueness proof draws on the proof of Theorem 1 in Schneider
  and Schneider \cite{SchSch91}.
\end{remark}
 Theorem \ref{the:vf} follows from Lemma
\ref{the:max-bal-g} and \eqref{eq:1a} of Lemma \ref{le:rep_1}.

The solution to equations \eqref{eq:56}
 is provided in Freidlin  and Wentzell \cite{wf2}.
Given $a\in A $\,, let $G_A(a)$ denote the set  of directed graphs
 that are in-trees with root $a$
on the vertex set $A$\,. Thus, for every
$a'\in A$\,, there is a unique directed path from $a'$ to
$a$ in $G_A(a)$\,. For $g\in
G_A(a)$\,, we let $E(g)$ denote the set of edges of $g$\,. Each edge
$e=(a',a'')\in E(g)$ 
is assigned the weight $v(e)=\Pi(a',a'')$\,. We let
$w(g)=\prod_{e\in E(g)}v(e)$\,. (If the set $E(g)$ is uncountable,
the latter product is defined as the infimum of finite products.)
\begin{lemma}
  \label{le:reshenie}
For $a\in A$\,,
\begin{equation*}
  \mathbf\Pi(a)=\frac{ \sup_{g\in
      G_A(a)}w(g)}{\sup_{a'\in A }
\sup_{g\in G_A(a')}w(g)}\,.
\end{equation*}
\end{lemma}
\begin{proof}
  Let $\{ A',\,  A''\}$ be a partition of
  $ A $\,. For  \eqref{eq:56}, it suffices to prove that
  \begin{equation}
    \label{eq:19}
    \sup_{a'\in A'}\sup_{a''\in A''}
\sup_{g\in
  G_{A}(a')}w(g)\Pi(a',a'')=
    \sup_{a'\in A'}\sup_{a''\in A''}
\sup_{g\in G_A(a'')}w(g)\Pi(a'',a')\,.
  \end{equation}
Let $g'\in G_A(a')$\,. Let
$(a''',a^{iv})$ be an edge on the path from
$a''\in A''$ to $a'\in A'$ in $g$ such that
$a'''\in  A''$ and $a^{iv}\in A'$\,. 
Let $g''$ represent the graph that is obtained from 
$g$ by inserting the edge $(a',a'')$ and
deleting the edge $(a''',a^{iv})$\,. We have
that $g''\in G_A(a''')$ and 
$w(g')\Pi(a',a'')
=w(g'')\Pi(a''',a^{iv})$\,.
Thus, each product on the lefthand side of \eqref{eq:19} is present on the righthand
side.
\end{proof}
\begin{remark}
  \label{re:ust}
There exists an  in-tree  $\overline g$
     with root $a$ on some subset of $A$ as the vertex set and no
     edges $(a',a'')$ with $\Pi(a',a'')=1$ such that $\sup_{g\in
        G_{A}(a)}w(g)=w(\overline g)$\,.
Indeed, if $g\in G_A(a)$ contains an edge $(a',a'')$ with $\Pi(a',a'')=1$\,,
let graph $g'\in G_A(a)$ be obtained from $g$
by replacing  each edge $(\tilde a,a')$ (with $a'$ as a terminal point)
 with the edge $(\tilde a,a'')$ (with $a''$ as a terminal point).
Then,  $\Pi(\tilde a,a')=\Pi(\tilde a,a')
\Pi(a',a'')\le \Pi(\tilde a,a'')$\,, so, $w(g)\le w(g')$\,.
With $\overline g$ being
obtained from $g'$ by deleting the edge $(a',a'')$\,, 
$w(\overline g)=w(g')$\,.  \end{remark}
\section{Diffusions with jumps}
\label{sec:diff-with-jumps}

In this section, we establish an LDP for the stationary distributions of
 jump diffusions.
Let us assume that  $X^n=(X^n_t,\,t\ge 0)$ 
is a semimartingale on a stochastic basis
$(\Omega,\mathcal{F},\mathbb{F}^n=(\mathcal{F}^n_t)_{t\ge 0},\mathbf P)$
that is a weak solution to the  equation 
\begin{equation}\label{eq:1.1}
X^n_t=x^n+\int_0^tb(X^n_s)ds+\frac{1}{\sqrt{n}}\int_0^t\sigma(X^n_s)dW^n_s
+ \frac{1}{n}\int_0^t\int_{\mathbb G}f(X^n_{s-},u)\bl(\mu^n(ds,du)-n\,
ds
\, \nu(du)\br),
\end{equation}
where $x^n\in \R^d$, $b(y)$ is an $\R^d$--valued Borel measurable
function,
$\sigma(y)$ is an $\R^{d\times m}$--valued Borel measurable function, 
$W^n_t$ is an $\R^m$--valued standard Wiener process,
$(\mathbb G,\mathcal{G})$ is a measurable space,
$f(y,u)$ is an $\R^d$--valued Borel measurable function, 
$\mu^n(ds,du)$ is a Poisson  random measure on $\R_+\times
\mathbb G$ with 
compensator $n\,ds\,\nu(du)$\,,
 $\nu $ being a $\sigma$--finite measure on $(\mathbb G,\mathcal{G})
   $\,, for the definitions, see, e.g., Gihman and Skorohod \cite{GihSko82}, Ikeda and Watanabe \cite{IkeWat}.

We assume 
the following conditions:
\begin{enumerate}
\item
there exists $C>0$  such that 
$\abs{b(y)}^2+\norm{\sigma(y)^\ast}^2
+\int_{\mathbb G}\abs{f(y,u)}^2\,\nu(du)\le C(1+\abs{y}^2)$\,,
with $^\ast$ denoting transpose,
\item the
functions $b(y)$ and $\sigma(y)$ are continuous and
 $\int_{\mathbb G}\abs{f(y',u)-f(y,u)}^2\,\nu(du)\to0$ when $y'\to y$\,.
\end{enumerate} 
Under those hypotheses, \eqref{eq:1.1} has  a  weak
solution, see Theorem 1 on p.357 in 
Gihman and Skorohod \cite{GihSko82}.
For   uniqueness,  one may require, in addition, that either
$b(y)$\,, $\sigma(y)$\,, and $f(y,u)$ 
be bounded and that $ \sigma(y)\sigma(y)^\ast$ be positive definite or 
that  the coefficients satisfy 
local Lipschitz continuity conditions, see
  Gihman and Skorohod \cite{GihSko82},
Ikeda and Watanabe \cite[Chapter 4]{IkeWat}, 
 Jacod and Shiryaev \cite[Chapter
III]{MR2003j:60001} and references therein.
 We do not require weak
uniqueness. Nevertheless, a certain nondegeneracy condition is
needed for other purposes. 
Specifically, we assume that
the matrix 
$ \sigma(y)\sigma(y)^\ast
$ is positive definite, for all $y$\,.

For a trajectorial LDP, we require that there
exists an $\R_+$--valued Borel function $h(u)$ such that
\begin{equation*}
  \int_{\mathbb G} (e^{\rho h(u)}-1-\rho h(u))\,\nu(du)<\infty\,,
\end{equation*}
for all $\rho\ge0$\,, and that
\begin{equation}
  \label{eq:67}
  \abs{f(y,u)}\le h(u)(1+\abs{y})\,.
\end{equation}
Suppose, in addition, that 
\begin{equation}
  \label{eq:24}
      \nu(u:\,e_i\cdot f(y,u)>0)>0\,,\nu(u:\,e_i\cdot f(y,u)<0)>0
\text{ for all }y\in\R^d\,,i\in\{1,2,\ldots,d\}\,,
\end{equation}
where $e_i$ denotes the $i$th vector of the standard basis in
$\R^d$\,, and that the following condition holds:
\begin{trivlist}
\item[(C)] for every $y\in\R^d$\,,
 there exists
  $\gamma>0$ 
 such that 
 \begin{equation*}
M=
\sup_{y':\,\abs{y-y'}\le\gamma} \sup_{u\in\mathbb G}\abs{f(y',u)-f(y,u)}<\infty\,.
    \end{equation*}
\end{trivlist}
If  $x^n\to x$\,, then   the sequence
$X^n$ obeys the LDP in $\D(\R_+,\R^d)$ with deviation function
\begin{equation}
  \label{eq:50}
    I_{x}(X)=\int_0^\infty\sup_{\lambda\in\R^d}(\lambda\cdot (\dot X_t-b(X_t))
-\frac{1}{2}\,\abs{\sigma(X_t)^\ast
\lambda}^2-
\int_{\mathbb G}\bl(e^{\lambda\cdot f(X_t,u)
    }-1-\lambda\cdot f(X_t,u)\br)\nu(du))\,dt\,,\end{equation}
provided $X=(X_t\,,t\ge0)$ is absolutely continuous
and $X_0=x$ and $I_{x}(X)=\infty$\,, otherwise.
 In some more detail, the LDP  follows from
Theorem 5.4.3 on p.419 in Puhalskii
\cite{Puh01} by   observing that
the predictable characteristics of $X^n$
without truncation, as defined, e.g., on p.292 in Puhalskii
\cite{Puh01},
are of the form  \begin{align*}
  B^n_t&=\int_0^tb(X_s^n)\,ds\,,
\\C^n_t&=\frac{1}{n}\,\int_0^t\sigma(X_s^n)\sigma(X^n_s)^\ast\,ds\,,
\\\nu^n((0,t], \Gamma )&=n\int_0^t
\int_{\mathbb G}\mathbf1_ {\Gamma\setminus\{0\}}\bl(\frac{f(X^n_s,u)}{n}\br)\nu(du)\,ds\,,
\end{align*}
where $\Gamma\in\mathcal{B}(\R^d)$\,,
so that the local characteristics, as on p.415 in Puhalskii
\cite{Puh01}, are given by
\begin{equation*}
  b^n_s(y)=b(y)\,, c^n_s(y)=\sigma(y)\sigma(y)^\ast\,,
  \nu^n_s( \Gamma , y)=\int_{\mathbb G}\mathbf1_{ \Gamma\setminus\{0\} }(f(y,u))\nu(du)\,.
\end{equation*}
In order to ensure the uniqueness of the idempotent distribution
$\mathcal{L}_i(X)$ required in the statement of 
Theorem 5.4.3 on p.419 in Puhalskii
\cite{Puh01}, one notes that, under condition (C),
 for $v>0$\,, $\eta>0$\,, $B>M$ and $\lambda\in\R^d$ with $\abs{\lambda}=1$\,,  if $\abs{y-y'}\le\gamma$
and 
$\eta\le 1-M/B$\,, then
\begin{equation}
  \label{eq:25}
    \int_{\mathbb G}e^{\eta v\lambda\cdot f(y,u)}\ind_{\{\lambda\cdot
    f(y,u)>B\}}\nu(du)
\le
  \int_{\mathbb G}e^{ v\lambda\cdot f(y',u)}
\ind_{\{\lambda\cdot f(y',u)>B-M\}}\nu(du)\,,
\end{equation}
and invokes Theorem 2.8.33 on p.244 
in Puhalskii \cite{Puh01} in analogy with the proof of
 Theorem 2.8.34 on p.245 in 
Puhalskii \cite{Puh01}.
 (In fact, it would be possible to merely refer to
the latter theorem,
as mentioned in the statement of 
Theorem 5.4.3 on p.419 in Puhalskii
\cite{Puh01}, if it allowed a different constant  to $B$\,, say $B_1$\,,
in the denominator  in condition 3b). That such a substitution
 is possible follows from the proof of 
Theorem 2.8.34 on p.245 in 
Puhalskii \cite{Puh01}. One would need to use the new version
with $B_1=B-M$\,.)

 By a standard argument, 
the function $I_x(X)$ is
lower semicontinuous in $X$\,. Since $I_x(X)=\infty$ unless
$X_0=x$\,,
$I_x(X)$  is lower
semicontinuous in $(x,X)$\,.
One can see that if  $I_x(X)$ is bounded above on  a set of $(x,X)$, 
then the functions
$X$ are locally equicontinuous uniformly in $x$\,. It follows that the
set $\cup_{x\in K}\{X:\,I_x(X)\le \gamma\}$ is  compact in
$\mathbb C(\R_+,\R^d)$ for any compact $K\subset \R^d$ and any
$\gamma\ge 0$\,. 
One checks  that  $
I_x(X)=0$ if and only if $X_0=x$ and
\begin{equation}
  \label{eq:52}
  \dot X_t=b(X_t)\,.
\end{equation}
We assume that there exists a locally finite collection 
$A$ of equilibria of \eqref{eq:52} which has a nonempty intersection
with the $\omega$--limit set of every $x\in\mathbb S$ and that
the function $b(x)$ is bounded on some neighbourhood of $A$\,.

Furthermore, we assume  that   $X^n$ admits  invariant measure
$\pi^n$\,. The latter property holds under various sets of
hypotheses, see, e.g., Masuda \cite{Mas07,Mas08},
Qiao \cite{Qia14}, Xie and Zhang \cite{XieZha20}.
A drift condition is commonly required.
We assume the following:
\begin{equation}
  \label{eq:62}
  y\cdot b(y)\le
-\kappa\abs{y}^2\,, \text{ provided $\abs{y}$ is great enough,
for some $\kappa>2C+4\int_{\mathbb G}h(u)^2e^{4h(u)}\nu(du)$\,.}
\end{equation} 

Let us   check assumptions {\bf 2.1 -- 2.4}
with $\mathbf\Pi_x(X)=e^{-I_x(X)}$\,.
We have already verified assumptions {\bf 2.1} and {\bf 2.2}.
Let us check {\bf 2.3}. By \eqref{eq:50},
\begin{equation*}
  \inf_{X'\in\pi_s^{-1}(\pi_sX)}I_x(X')=\int_0^s\sup_{\lambda\in\R^d}(\lambda\cdot (\dot X_t-b(X_t))
-\frac{1}{2}\,\abs{\sigma(X_t)^\ast
\lambda}^2-
\int_{\mathbb G}\bl(e^{\lambda\cdot f(X_t,u)
    }-1-\lambda\cdot f(X_t,u)\br)\nu(du))dt
\end{equation*}
so that
\eqref{eq:3} holds. 
 In {\bf 2.4}, only conditions (4) and (5)  are not evident.
Given $x,x'\in\R^d$\,, let $X_t=x+t(x'-x)/\abs{x'-x}$\,.
We have that $X_0=x$\,, $X_{\abs{x'-x}}=x'$ and
\begin{multline*}
  \int_0^{\abs{x'-x}}\sup_{\lambda\in\R^d}(\lambda\cdot (\dot X_t-b(X_t))
-\frac{1}{2}\,\abs{\sigma(X_t)^\ast
\lambda}^2-
\int_{\mathbb G}\bl(e^{\lambda\cdot f(X_t,u)
    }-1-\lambda\cdot f(X_t,u)\br)\nu(du))\,dt\\
\le  \int_0^{\abs{x'-x}}\sup_{\lambda\in\R^d}(\lambda\cdot (\dot X_t-b(X_t))
-\frac{1}{2}\,\abs{\sigma(X_t)^\ast
\lambda}^2)\,dt\,,
\end{multline*}
which implies parts (4) and (5) in {\bf 2.4}. (Incidentally, the
hypotheses of  part (3) can be
checked similarly.)
 
We precede the check of {\bf 2.5} with a proof of the exponential tightness
of the $\pi^n$\,.
\begin{lemma}
  The sequence $\pi^n$ is exponentially tight of order $n$\,.
\end{lemma}
\begin{proof}
Let us prove that  
\begin{equation}
  \label{eq:66}
  \lim_{L\to\infty}\limsup_{n\to\infty}\limsup_{t\to\infty}\mathbf P(\abs{X^n_t}>L)^{1/n}=0\,.
\end{equation}
Let
\begin{equation*}
  M^n_t=\frac{1}{\sqrt{n}}\int_0^t\sigma(X^n_s)dW^n_s
+ \frac{1}{n}\int_0^t\int_{\mathbb G}f(X^n_{s-},u)\bl(\mu^n(ds,du)-n\, ds\, \nu(du)\br)\,.
\end{equation*}
It is a locally square integrable martingale with
$C^n_t$ as the predictable
quadratic variation process of the continuous part.
As in Liptser and Pukhalskii \cite{lp}, we look at $\abs{X^n_t}^{2n}$\,.
By It\^o's lemma, see, e.g., p.57 in Jacod and Shiryaev
\cite{MR2003j:60001}, with tr standing for the trace of a matrix,
\begin{multline*}
\abs{X^n_t}^{2n}
=\abs{x^n}^{2n}+
\int_0^t2n\abs{X^n_s}^{2(n-1)} X^n_s \cdot b(X^n_s)\,ds+
\int_0^t2n\abs{X^n_{s-}}^{2(n-1)} X^n_{s-}\cdot dM^n_s\\+
\frac{1}{2n}\,\int_0^t
\text{tr}\bl(
2n2(n-1)\abs{X^n_s}^{2(n-2)} X^n_s{X^n_s}^\ast\sigma(X^n_s)\sigma(X^n_s)^\ast
+2n\abs{X^n_s}^{2(n-1)} \,\sigma(X^n_s)\sigma(X^n_s)^\ast\br)ds\\
+\sum_{0<s\le t}\bl(\abs{X^n_s}^{2n}-\abs{X^n_{s-}}^{2n}
-2n\abs{X^n_{s-}}^{2(n-1)} X^n_{s-}\cdot \Delta
X^n_s\br)\\
=\abs{x^n}^{2n}+
2n\int_0^t\abs{X^n_s}^{2(n-1)} X^n_s \cdot b(X^n_s)\,ds+
2n\int_0^t\abs{X^n_{s-}}^{2(n-1)} X^n_{s-}\cdot dM^n_s\\+
\int_0^t
\text{tr}\bl(2(n-1)\abs{X^n_s}^{2(n-2)} X^n_s{X^n_s}^\ast\sigma(X^n_s)\sigma(X^n_s)^\ast
+\abs{X^n_s}^{2(n-1)} \sigma(X^n_s)\sigma(X^n_s)^\ast\br)ds\\
+\int_0^t\int_{\mathbb G}\bl(\abs{X^n_{s-}+\frac{f(X^n_{s-},u)}{n}}^{2n}
-\abs{X^n_{s-}}^{2n}
-2n\abs{X^n_{s-}}^{2(n-1)} X^n_{s-}\cdot
\frac{f(X^n_{s-},u)}{n}\br)\mu^n(ds,du)\,.
\end{multline*}
By Taylor's formula, 
\begin{multline*}
\abs{X^n_{s-}+\frac{f(X^n_{s-},u)}{n}}^{2n}
-\abs{X^n_{s-}}^{2n}
-2n\abs{X^n_{s-}}^{2(n-1)} X^n_{s-}\cdot
\frac{f(X^n_{s-},u)}{n}\\
=
\int_0^1 (1-t)
\bl(\frac{4(n-1)}{n}\,\abs{X^n_{s-}+t\,\frac{ f(X^n_{s-},u)}{n}}^{2(n-2)}
\abs{(X^n_{s-}+t\,\frac{ f(X^n_{s-},u)}{n})\cdot
f(X^n_{s-},u)}^2\\+
\frac{2}{n}\,\abs{X^n_{s-}+t\,\frac{ f(X^n_{s-},u)}{n}}^{2n-2}
\abs{f(X^n_{s-},u)}^2\br)\,dt
\le 
2(\abs{X^n_{s-}}+\frac{\abs{ f(X^n_{s-},u)}}{n})^{2(n-1)}
\abs{f(X^n_{s-},u)}^2
\\\le 2(\abs{X^n_{s-}}\vee1)^{2(n-1)}e^{2\abs{
    f(X^n_{s-},u)}/(\abs{X^n_{s-}}\vee1)}
\abs{f(X^n_{s-},u)}^2\,.
\end{multline*}
By the linear growth condition \eqref{eq:67},
$\abs{
    f(X^n_{s-},u)}/(\abs{X^n_{s-}}\vee1)\le2 h(u)$\,.
It follows that, with $F_t\prec G_t$ meaning that $G_t-F_t$ is a
nondecreasing function,
\begin{multline*}
  \abs{X^n_t}^{2n}\prec\abs{x^n}^{2n}+
2n\int_0^t\abs{X^n_s}^{2(n-1)} X^n_s \cdot b(X^n_s)\,ds+
2n\int_0^t\abs{X^n_{s-}}^{2(n-1)} X^n_{s-}\cdot dM^n_s\\+
\int_0^t
\bl(
2(n-1)\abs{X^n_s}^{2(n-2)}\abs{\sigma(X^n_s)^\ast X^n_s}^2
+\abs{X^n_s}^{2(n-1)} \text{tr}(\sigma(X^n_s)\sigma(X^n_s)^\ast)\br)\,ds\\
+2\int_0^t\int_{\mathbb G}(\abs{X^n_{s-}}\vee1)^{2(n-1)}
\abs{f(X^n_{s-},u)}^2e^{4h(u)}\mu^n(ds,du)\,.
\end{multline*}
Hence,
\begin{multline*}
  \abs{X^n_t}^{2n}\prec\abs{x^n}^{2n}+
2n\int_0^t\abs{X^n_s}^{2(n-1)} X^n_s \cdot b(X^n_s)\,ds+
\int_0^t
\bl(
2(n-1)\abs{X^n_s}^{2(n-2)}\abs{\sigma(X^n_s)^\ast X^n_s}^2
\\+\abs{X^n_s}^{2(n-1)} \text{tr}(\sigma(X^n_s)\sigma(X^n_s)^\ast)\br)\,ds
+2n\int_0^t\int_{\mathbb G}(\abs{X^n_{s-}}\vee1)^{2(n-1)}
\abs{f(X^n_{s-},u)}^2e^{4h(u)}\nu(du)\,ds+\overline{M}^n_t\,,
\end{multline*}
where $\overline{ M}^n_t$ is a local martingale.
By \eqref{eq:62} and the linear growth conditions in 1. and \eqref{eq:67}, for some $C_1>0$\,, $C_2>0$ and $R>0$\,,
\begin{equation*}
    \abs{X^n_t}^{2n}\prec\abs{x^n}^{2n}+
nC_1R^{2n}t-nC_2\int_0^t\abs{X^n_s}^{2n}\,ds
+\overline{M}^n_t\,.
\end{equation*}
Let $\tau^n_k\,, k=1,2,\ldots,$ represent a localising sequence for $\overline M^n$\,. We have that
\begin{equation*}
      \abs{X^n_{t\wedge \tau^n_k}}^{2n}\prec\abs{x^n}^{2n}+
nC_1R^{2n}(t\wedge\tau^n_k)-nC_2\int_0^{t\wedge\tau^n_k}\abs{X^n_s}^{2n}\,ds
+\overline{M}^n_{t\wedge\tau^n_k}\,,
\end{equation*}
so,
\begin{equation*}
      \mathbf E\abs{X^n_{t\wedge \tau^n_k}}^{2n}\prec\abs{x^n}^{2n}+
nC_1R^{2n}\mathbf E(t\wedge\tau^n_k)-
nC_2\mathbf E\int_0^{t\wedge\tau^n_k}\abs{X^n_s}^{2n}\,ds\,.
\end{equation*}
On letting $k\to\infty$\,,
\begin{equation*}
    \mathbf E\abs{X^n_t}^{2n}\prec\abs{x^n}^{2n}+
nC_1R^{2n}t-nC_2\int_0^t\mathbf E\abs{X^n_s}^{2n}\,ds
\,.
\end{equation*}
Hence, 
\begin{equation*}
  \mathbf E\abs{X^n_t}^{2n}\le e^{-nC_2t}(x^n)^{2n}+R^{2n}\frac{C_1}{C_2}
\end{equation*}
so that
\begin{equation}
  \label{eq:59}
 \limsup_{n\to\infty} \limsup_{t\to\infty}(\mathbf
 E\abs{X^n_t}^{2n})^{1/n}\le R^2\,,
\end{equation}
implying \eqref{eq:66}.

By \eqref{eq:66} and exponential Markov's inequality,
\begin{equation*}
\lim_{L\to\infty}\limsup_{n\to\infty}\pi^n(x\in\R^d:\,
\abs{x}>L)^{1/n}=0\,,
\end{equation*}
 so, 
the sequence $\pi^n$ is exponentially tight of order $n$\,.
\end{proof}
The tightness of $\Pi_{x,t}$ required in {\bf 2.5} follows from \eqref{eq:59}: ''by
Fatou'' and the distributions of the $X^n_t$ LD converging to $\Pi_{x,t}$\,,
\begin{equation*}
\liminf_{n\to\infty}
\sup_{x\in K}\limsup_{t\to\infty}  \bl(\mathbf E\abs{  X^n_t}^{2n}
\br)^{1/n}\ge\sup_{x\in K}\limsup_{t\to\infty} \sup_{y\in\R^d}
  \abs{y}^2\Pi_{x,t}(y)\,,
\end{equation*}
 so, we obtain from \eqref{eq:59} that
  \begin{equation*}
\lim_{L\to\infty}    \sup_{x\in K}\limsup_{t\to\infty}\Pi_{x,t}(\abs{y}>L)\le
    \lim_{L\to\infty}\sup_{x\in K}\limsup_{t\to\infty}
L^{-2}\sup_{y\in \R^d}
  \abs{ y}^2\Pi_{x,t}(y)=0\,.
  \end{equation*}

 Let, for $a,\tilde a\in A$\,,
\begin{equation}
  \label{eq:13}
    I(a,\tilde a)=\lim_{T\to\infty}\inf_{\substack{X\in \C(\R_+,\R^d):\\
\,X_0=a,\,X_T=\tilde a}}I_{a}(X)\,,
\end{equation}
the limit existing by Lemma \ref{le:attraction}. Theorem
\ref{the:vf} yields the following result.
\begin{theorem}
\label{the:jump_diff_ld}  Under the stated hypotheses, the measures $\pi^n$ obey the LDP
in $\R^d$  for rate $n$   with deviation function $I(x)$\,, which is
  specified uniquely by the requirements that 
$I(x)=\inf_{a\in A}(I(a)+I(a,x))$\,, 
$\inf_{a\in A}I(a)=0$\,, and 
  \begin{equation*}
    \inf_{a'\in A'}\inf_{a''\in A''}\bl(I(a')+I(a',a'')\br)=
\inf_{a''\in A''}\inf_{a'\in A'}\bl(I(a'')+I(a'',a')\br)\,,
  \end{equation*}
for arbitrarily chosen partitions $\{A',A''\}$ of  $A$\,.
The $I(a)\,, a\in A$\,, can be calculated as follows:
\begin{equation*}
    I(a)=\inf_{g\in
      G_A(a)}\sum_{(a',a'')\in E(g)}I(a',a'')-\inf_{\tilde a\in A }
\inf_{g\in G_A(\tilde a)}\sum_{(a',a'')\in E(g)}I(a',a'')\,.
\end{equation*}
\end{theorem}
\begin{remark}
In \eqref{eq:13},
 it can be assumed that $X_t$ follows \eqref{eq:52} for
$t\ge T$\,, so, the integration in \eqref{eq:50} can be stopped at $T$\,. Also,
\begin{equation*}
  I(a,\tilde a)=\inf_{\substack{X\in \C(\R_+,\R^d)\,, T\ge
    0:\\\,X_0=a,\,X_T=\tilde a}}I_{a}(X)=
\inf_{\substack{X\in \C(\R_+,\R^d)\,:\,X_0=a,\,\\X_t\to \tilde a\,\text{ as }t\to\infty}}I_{a}(X)\,.
\end{equation*}
For diffusion processes,
the lefthand representation of $I(a,\tilde a)$  was used in Freidlin and
Wentzell \cite{wf2}.
\end{remark}

We now look at  moderate deviation setups. Let
\begin{multline*}
  X^{m,n}_t=x^{m,n}+
\int_0^tb(X^{m,n}_s)ds+\frac{1}{\sqrt{m}}\int_0^t\sigma(X^{m,n}_s)dW^{n}_s\\
+ \frac{1}{\sqrt{nm}}\int_0^t\int_{\mathbb G}f(X^{m,n}_{s-},u)\bl(\mu^{n}(ds,du)-n\, ds\,
\nu(du)\br)\,, 
\end{multline*}
where $n\to\infty$\,, $m\to\infty$\,,  $n/m\to\infty$ and 
''the primitive data''  are as above.  

Let us assume that one of the following 
conditions holds:
 \begin{trivlist}
\item[$(\text{P})$]
for some $\delta>0$\,,
\begin{equation*}
  \int_{\mathbb G} |f(y,u)|^{2+\delta}\,\nu(du)
<\infty\,,\; y\in\R^d\,,
\end{equation*}
and $\ln n/m\to\infty$\,;
\item[$(\text{SE})$] for some $\beta\in (0,1]$ and 
$\alpha>0$\,,
\begin{equation*}
  \int_{\mathbb G}e^{\alpha|f(y,u)|^\beta}\,\nu(du)<\infty\,,\; y\in\R^d\,,
\end{equation*}
and $n^\beta/m^{2-\beta}\to\infty$\,.
\end{trivlist} Let also $x^{m,n}\to x$\,. Then, the LDP  holds for $X^{m,n}$ in
$\D(\R_+,\R^d)$ for rate
$m$ with deviation function
\begin{equation*}
 \hat I_x(X)=\frac{1}{2}\,\int_0^\infty (\dot X_t-b(X_t))\cdot
\bl(\sigma(X_t)\sigma(X_t)^\ast+\int_{\mathbb G}f(X_t,u)f(X_t,u)^\ast\,\nu(du)\br)^{-1}
(\dot X_t-b(X_t))\,dt\,,
\end{equation*}
provided $X_t$ is absolutely continuous and $X_0=x$\,, and
$I_x(X)=\infty$\,, otherwise. The proof is done by applying Theorem
5.4.4 on p.423 in Puhalskii \cite{Puh01} (with $\alpha_\phi=n$ and $\beta_\phi=\sqrt{nm})$\,. The other hypotheses of Theorem
\ref{the:vf} are checked as for $X^n$\,, e.g.,
$\limsup_{n,m\to\infty,\,n/m\to\infty}\limsup_{t\to\infty}(\mathbf
E\abs{X^{m,n}_t}^{2m})^{1/m}<\infty$\,. 
It follows that the analogue  of Theorem \ref{the:jump_diff_ld} holds
for the $X^{m,n}$ and rate $m$\,, 
with the stationary distribution of $X^{m,n}$ and $\hat I_x$
substituted for $\pi^n$ and $I_x$\,, respectively.

As another example, we consider stationary 
moderate deviations around the equilibria
of \eqref{eq:52}. Let $\tilde X^{n,m}_t=\sqrt{n/m}\,(X^n_t-x)$\,, where
$b(x)=0$  and  $n/m\to\infty$\,.
Suppose that one of the conditions (P)
 or  (SE) holds.
Let $\sqrt{n/m}\,(x^{n}-x)
\to\tilde x$ and $b(y)$ be continuously differentiable.
Then, by Theorem 5.4.4 on p.423 in Puhalskii \cite{Puh01}, the $\tilde
X^{n,m}$ obey the LDP for rate $m$ in $\D(\R_+,\R^d)$ with deviation
function 
\begin{equation*}
  \tilde I_{\tilde x}(X)=\frac{1}{2}\,\int_0^\infty (\dot X_t-
D b(x)X_t)\cdot
c(x)^{-1}
(\dot X_t-D b(x)X_t)\,dt\,,
\end{equation*}
provided $X_t$ is absolutely continuous and $X_0=\tilde x$\,, and
$\tilde I_{\tilde x}(X)=\infty$\,, otherwise, where $Db(y)$ stands for
the derivative of $b(y)$ and \begin{equation*}
  c(x)=\sigma( x)\sigma( x)^\ast
+\int_{\mathbb G}f( x,u)f( x,u)^\ast\,\nu(du)\,.
\end{equation*}
Let us assume 
that the matrix $Db( x)$ is stable. 
Then, the solutions of the equation
 $\dot X_t=Db(x)X_t$ converge to 0 so that the deviation function is
 given by a quasipotential.
Then, an analogue of Theorem
\ref{the:jump_diff_ld} for the stationary distributions of $\tilde X^{n,m}_t$ holds.

The quasipotential   can be evaluated explicitly.
Let us  look for $\inf\tilde I_0(X)$ over $X\in \C(\R_+,\R^d)$ and 
$T\in\R_+$ such that $X_{T}=r\in \R^d$\,.
The Hamiltonian associated with the Lagrangian $L(y,y')=(y'-
D b( x)y)\cdot c(x)^{-1}
(y'-D b( x)y)/2$ 
 is given by
\begin{equation*}
  H(y,p)=\sup_{y'\in \R^d}(p\cdot y'-L(y,y'))
=\frac{1}{2}\,p\cdot c(x) p+
p\cdot Db( x)y\,.
\end{equation*}
The Hamilton equations 
 $\dot p_t^{(T)}=-H_y(X^{(T)}_{t},p^{(T)}_t)$ and
 $\dot{X}^{(T)}_t=
H_p(X^{(T)}_t,p^{(T)}_t)$
along with the terminal condition $X^{(T)}_T=r$ 
yield
 $p^{(T)}_t=e^{-Db( x)^\ast t}C^{(T)}$
and 
\begin{equation*}
  X^{(T)}_t
=
e^{Db( x)t}\int_0^te^{-Db( x)s}c(x)e^{-Db( x)^\ast s}\,ds\,C^{(T)}\,, 
\end{equation*}
where
\begin{equation*}
  C^{(T)}=\bl(e^{Db( x)T}\int_0^{T}e^{-Db( x)t}c(x)e^{-Db( x)^\ast t}\,dt\br)^{-1}
r\,.
\end{equation*}
 Therefore, for the optimal trajectory,
\begin{equation*}
  \tilde I_{0}(X^{(T)})
=
\frac{1}{2}\,\int_0^{T}p^{(T)}_t\cdot c(x)p^{(T)}_t\,dt\\
\\
=\frac{1}{2}\,r^\ast
e^{-Db^\ast( x)T}\bl(\int_0^{T}e^{-Db( x)t}c(x)e^{-Db( x)^\ast t}\,dt\br)^{-1}
e^{-Db( x)T}r
\,.
\end{equation*}
The infimum over ${T}>0$ is
\begin{equation*}
  \frac{1}{2}\,r^\ast
\bl(\int_0^\infty
e^{Db( x)t}c(x)e^{Db( x)^\ast t}\,dt\br)^{-1}
r
\,,
\end{equation*}
 which is the deviation function for the stationary distributions of
the $\tilde X^{n,m}_t$\,.
One can also see that, for $t\ge0$\,,  as $T\to\infty$\,,
\begin{equation*}
  X^{(T)}_{T-t}\to \bl(\int_0^\infty e^{Db(x)s}c(x)e^{Db(x)^\ast
    s}ds\br)e^{Db(x)^\ast t}
\bl(\int_0^\infty e^{Db(x)s}c(x)e^{Db(x)^\ast s}ds\br)^{-1}r\,.
\end{equation*}
\appendix 

\section{On continuous LD convergence}\label{cont_conv}
The next  lemma   concerns the requirement in
condition {\bf 2.1} that $x^n$ should belong to the support of $P^n$
and is used in the proof of the main result.
\begin{lemma}
  \label{le:cont}
Let $S$ be a metric space and $P_n$ be a sequence of probability
measures
on the Borel $\sigma$-algebra of $S$\,. Suppose that the sequence $P_n$
LD converges at rate $n$ to deviability $\Pi$ on $S$\,.
 Let $h_n$  be 
$\R_+$--valued bounded Borel functions on $S$ and let $h$ be an
$\R_+$--valued function on $S$\,.
If 
$\limsup_{n\to\infty}h_n(y_n)\le h(y)$  for every sequence $y_n$ of
elements of $S$ and every $y$ such that $y_n\to y$\,,
 $y_n$ belongs to the support of $P_n$\,, and $\Pi(y)>0$\,, then
 \begin{equation*}
   \limsup_{n\to\infty}\bl(\int_Sh_n(y)^n\,P_n(dy)\br)^{1/n}\le
\sup_{y\in S}h(y)\Pi(y)\,.
 \end{equation*}
If, rather, $\liminf_{n\to\infty}h_n(y_n)\ge h(y)$\,, then
 \begin{equation*}
   \liminf_{n\to\infty}\bl(\int_Sh_n(y)^n\,P_n(dy)\br)^{1/n}\ge
\sup_{y\in S}h(y)\Pi(y)\,.
 \end{equation*}
\end{lemma}
\begin{proof}
Let $\mathcal{U}_y$ represent the collection of open neighbourhoods of
$y\in S$\,, let $F_n$ denote the support of $P_n$\,, let
\begin{align*}
  \overline h_n(y)=\inf_{U\in\mathcal{U}_y}
\sup_{m\ge n}\sup_{y'\in U\cap
  F_m}
h_m(y')
\intertext{and let}
\overline h(y)=\inf_n\overline h_n(y)\,.
\end{align*}
The functions $\overline h_n$ and $\overline h$ are upper
semicontinuous, $\overline h_n\downarrow \overline h$\,, as
$n\to\infty$\,, and $\overline h\le h$ $\Pi$--a.e.
Then, on choosing $n_0$ such that
$\sup_{y\in S}\overline h_n(y)\Pi(y)\le
\sup_{y\in S}\overline h(y)\Pi(y)+\epsilon$\,, for $n\ge n_0$\,, and noting that 
$h_n(y)\le\overline h_{n_0}(y)$\,, for $n\ge n_0$ and $y\in F_n$\,,
with the use of Theorem 3.1.3 on p.254 in Puhalskii \cite{Puh01},
\begin{multline*}
  \limsup_{n\to\infty}\bl(\int_S h_n(y)^n\,P_n(dy)\br)^{1/n}
=\limsup_{n\to\infty}\bl(\int_{F_n} h_n(y)^n\,P_n(dy)\br)^{1/n}
\le\limsup_{n\to\infty}\bl(\int_{F_n}\overline
  h_{n_0}(y)^n\,P_n(dy)\br)^{1/n}\\\le
  \limsup_{n\to\infty}\bl(\int_S\overline
  h_{n_0}(y)^n\,P_n(dy)\br)^{1/n}\le
\sup_{y\in S}\overline h_{n_0}(y)\Pi(y)
\le \sup_{y\in S}\overline h(y)\Pi(y)+\epsilon\le
\sup_{y\in S} h(y)\Pi(y)+\epsilon\,.
\end{multline*}
The second part is proved similarly,  on introducing
$  \underline h_n(y)=\sup_{U\in\mathcal{U}_y}
\inf_{m\ge n}\inf_{y'\in U\cap
  F_m}
h_m(y')$ and
$\underline h(y)=\sup_n\underline h_n(y)\,.$
\end{proof}

\def\cprime{$'$} \def\cprime{$'$} \def\cprime{$'$} \def\cprime{$'$}
  \def\cprime{$'$} \def\polhk#1{\setbox0=\hbox{#1}{\ooalign{\hidewidth
  \lower1.5ex\hbox{`}\hidewidth\crcr\unhbox0}}} \def\cprime{$'$}
  \def\cprime{$'$} \def\cprime{$'$} \def\cprime{$'$} \def\cprime{$'$}
  \def\cprime{$'$}


\begin{thebibliography}{10}

\bibitem{Aub93}
J.-P. Aubin.
\newblock {\em Optima and equilibria}, volume 140 of {\em Graduate Texts in
  Mathematics}.
\newblock Springer-Verlag, Berlin, 1993.
\newblock An introduction to nonlinear analysis, Translated from the French by
  Stephen Wilson.

\bibitem{wf2}
{M.I.} Freidlin and {A.D.} Wentzell.
\newblock {\em Random Perturbations of Dynamical Systems}.
\newblock Springer, 2nd edition, 1998.

\bibitem{GihSko82}
I.I. Gihman and A.V. Skorohod.
\newblock {\em Stokhasticheskie differentsialnye uravneniya i ikh
  prilozheniya}.
\newblock Naukova Dumka, 1982.
\newblock (in Russian).

\bibitem{IkeWat}
N.~Ikeda and S.~Watanabe.
\newblock {\em Stochastic Differential Equations and Diffusion Processes}.
\newblock North Holland, 2nd edition, 1989.

\bibitem{MR2003j:60001}
J.~Jacod and A.N. Shiryaev.
\newblock {\em Limit Theorems for Stochastic Processes}, volume 288 of {\em
  Grundlehren der Mathematischen Wissenschaften [Fundamental Principles of
  Mathematical Sciences]}.
\newblock Springer-Verlag, Berlin, second edition, 2003.

\bibitem{lp}
R.Sh. Liptser and A.~Pukhalskii.
\newblock Limit theorems on large deviations for semimartingales.
\newblock {\em Stoch. Stoch. Rep.}, 38:201--249, 1992.

\bibitem{Mas07}
H.~Masuda.
\newblock Ergodicity and exponential {$\beta$}-mixing bounds for
  multidimensional diffusions with jumps.
\newblock {\em Stochastic Process. Appl.}, 117(1):35--56, 2007.
\newblock Erratum: Stochastic Process. Appl., 119 (2009), 676--678.

\bibitem{Mas08}
H.~Masuda.
\newblock On stability of diffusions with compound-{P}oisson jumps.
\newblock {\em Bull. Inform. Cybernet.}, 40:61--74, 2008.

\bibitem{Puh97}
A.~Puhalskii.
\newblock Large deviations of semimartingales: a maxingale problem approach.
  {I}. {L}imits as solutions to a maxingale problem.
\newblock {\em Stoch. Stoch. Rep.}, 61:141--243, 1997.

\bibitem{Puh01}
A.~Puhalskii.
\newblock {\em Large Deviations and Idempotent Probability}.
\newblock Chapman \& Hall/CRC, 2001.

\bibitem{Puh03}
A.~Puhalskii.
\newblock On large deviation convergence of invariant measures.
\newblock {\em J. Theoret. Probab.}, 16(3):689--724, 2003.

\bibitem{Puh19a}
A.~Puhalskii.
\newblock Large deviations of the long term distribution of a non {M}arkov
  process.
\newblock {\em Electron. Commun. Probab.}, 24:Paper No. 35, 11, 2019.

\bibitem{Qia14}
H.~Qiao.
\newblock Exponential ergodicity for {SDE}s with jumps and non-{L}ipschitz
  coefficients.
\newblock {\em J. Theoret. Probab.}, 27(1):137--152, 2014.

\bibitem{SchSch91}
H.~Schneider and M.H. Schneider.
\newblock Max-balancing weighted directed graphs and matrix scaling.
\newblock {\em Math. Oper. Res.}, 16(1):208--222, 1991.

\bibitem{XieZha20}
L.~Xie and X.~Zhang.
\newblock Ergodicity of stochastic differential equations with jumps and
  singular coefficients.
\newblock {\em Ann. Inst. Henri Poincar\'{e} Probab. Stat.}, 56(1):175--229,
  2020.

\end{thebibliography}
\end{document}